\documentclass{article}
\usepackage{bm}
\usepackage{tikz,mathpazo}
\usepackage{pgf}
\usepackage[top=2.5cm, bottom=2.5cm, left=3cm, right=3cm]{geometry}   
\usepackage{indentfirst}
\usepackage{graphicx}
\usepackage{graphics}
\usepackage[toc,page,title,titletoc,header]{appendix}
\usepackage{bm}
\usepackage{bbm}
\usepackage{listings}  
\usepackage{amsmath}
\usepackage{setspace} 
\usepackage{indentfirst}
\usepackage{caption}
\usepackage{multirow} 
\usepackage{lipsum,multicol}
\usepackage{pdfpages}
\usepackage{float}
\usepackage{amsthm}
\usepackage{pdfpages}
\usepackage{url}
\usepackage{colortbl}
\usepackage{subfigure}
\usepackage{epsfig}
\usepackage{epstopdf}
\usetikzlibrary{shapes.geometric, arrows}
\usepackage{fancyhdr}
\usepackage{abstract}
\usepackage{tikz,mathpazo}
\usetikzlibrary{shapes.geometric, arrows}
\usepackage{amssymb}
\usepackage{latexsym}
\usepackage{verbatim}
\usepackage[numbers]{natbib} 
\usepackage{booktabs}

\usepackage{hyperref}
\hypersetup{
    colorlinks=true,
    linkcolor=blue,
    filecolor=magenta,      
    urlcolor=blue,
    citecolor=blue,
    }

\allowdisplaybreaks[4]

\newcommand{\inner}[1]{\left\langle #1 \right\rangle}
\newcommand{\norm}[1]{\left\Vert #1\right\Vert}
\newcommand{\bb}[1]{\mathbb{#1}}

\newcommand{\X}{{ \ca{X} }}
\newcommand{\Y}{{ \ca{Y} }}
\newcommand{\Z}{{ \ca{Z} }}

\newcommand{\K}{{ \ca{K} }}
\newcommand{\M}{{ \ca{M} }}
\newcommand{\E}{{ \ca{E} }}

\newcommand{\ca}[1]{\mathcal{#1}}

\newcommand{\Diag}[0]{\mathrm{Diag}}
\newcommand{\ff}{_{\mathrm{F}}}

\newcommand{\tp}{^\top}

\newcommand{\TX}{{\ca{T}_{\X}}}
\newcommand{\NX}{\ca{N}_{\X}}

\newcommand{\xk}{{x_{k} }}

\newcommand{\xkp}{{x_{k+1} }}

\newcommand{\A}{\ca{A}}
\newcommand{\RA}{\ca{R}_{\A}}
\newcommand{\Ja}{{\nabla \A}}
\newcommand{\Jc}{{\nabla c}}
\newcommand{\DJa}{{\nabla^2 \A}}
\newcommand{\DJc}{{\nabla^2 c}}

\newcommand{\Rn}{\mathbb{R}^n}
\newcommand{\Rp}{\mathbb{R}^p}
\newcommand{\Rm}{\mathbb{R}^m}

\newcommand{\Rq}{\mathbb{R}^q}

\newcommand{\Mxb}{{M_{x,B}}}
\newcommand{\Mxf}{{M_{x,f}}}
\newcommand{\Mxc}{{M_{x,c}}}
\newcommand{\Mxa}{{M_{x,A}}}
\newcommand{\Mxr}{{M_{x,R}}}
\newcommand{\Mxi}{{M_{x,I}}}
\newcommand{\Mxjc}{{M_{x,P}}}

\newcommand{\Lxa}{{L_{x,A}}}

\newcommand{\Lxf}{{L_{x,f}}}

\newcommand{\Lxres}{ {L_{x,q}} }

\newcommand{\sigmaxc}{ \sigma_{x,c} }

\newcommand{\Omegax}[1]{ {\Omega_{#1}} }

\newtheorem{theo}{Theorem}[section]
\newtheorem{lem}[theo]{Lemma}
\newtheorem{prop}[theo]{Proposition}

\newtheorem{coro}[theo]{Corollary}

\newtheorem{defin}[theo]{Definition}
\newtheorem{rmk}[theo]{Remark}
\newtheorem{assumpt}[theo]{Assumption}

\usepackage[figuresright]{rotating}
\usepackage{pdflscape}

\usepackage{caption}
\usepackage{algorithm}
\usepackage{algpseudocode}
\usepackage{longtable}
\usepackage{appendix}
\usepackage{tikz-cd}

\numberwithin{equation}{section}

\usepackage{array}

\title{An Exact Penalty Approach for Equality Constrained Optimization over a Convex Set} 
\author{
Nachuan Xiao\thanks{School of Data Science, The Chinese University of Hong Kong, Shenzhen, China. (xncxy@cuhk.edu.cn).},~
Tianyun Tang\thanks{Institute of Operational Research and Analytics, National University of Singapore, Singapore. (ttang@nus.edu.sg).},~ 
Shiwei Wang\thanks{Institute of Operational Research and Analytics, National University of Singapore, Singapore. (wangshiwei@amss.ac.cn).},~
Kim-Chuan Toh\thanks{Department of Mathematics, and Institute of Operations Research and Analytics, National University of Singapore, Singapore 119076. (mattohkc@nus.edu.sg).}}

\begin{document}
    \maketitle
    
    \begin{abstract}
        In this paper, we consider the nonlinear constrained optimization problem (NCP) with constraint set $\{x \in \X: c(x) = 0\}$, where $\X$ is a closed convex subset of $\mathbb{R}^n$. We propose an exact penalty approach, named constraint dissolving approach, that transforms (NCP) into its corresponding constraint dissolving problem (CDP). The transformed problem (CDP) admits $\X$ as its feasible region with a locally Lipschitz smooth objective function. We prove that (NCP) and (CDP) share the same first-order stationary points, second-order stationary points, second-order sufficient condition (SOSC) points, and strong SOSC points, in a neighborhood of the feasible region. Moreover, we prove that these equivalences extend globally under a particular error bound condition. Therefore, our proposed constraint dissolving approach enables direct implementations of optimization approaches over $\X$ and inherits their convergence properties to solve problems that take the form of (NCP). Preliminary numerical experiments illustrate the high efficiency of directly applying existing  solvers for optimization over $\X$ to solve (NCP) through (CDP). These numerical results further demonstrate the practical potential of our proposed constraint dissolving approach. 
    \end{abstract}

    \section{Introduction}
    In this paper, we consider the following nonlinear constrained optimization problem (NCP), 
    \begin{equation}
        \label{Prob_Ori}
        \tag{NCP}
        \begin{aligned}
              \min_{x \in \Rn}\quad & f(x),\\
                \text{s. t.} \quad &c(x) = 0, \quad x \in \X,
        \end{aligned}
    \end{equation}
    where $\X$ is a closed convex subset of $\Rn$. We denote $\K := \{x \in \Rn: c(x) = 0, ~x \in \X \}$ as the feasible region of \eqref{Prob_Ori}, and make the Assumption \ref{Assumption_f} on \eqref{Prob_Ori} throughout this paper. 
    Before presenting the assumptions, we introduce some necessary notation.
                We use $\mathrm{aff}(\X)$ to denote the affine hull of $\X$, while $\ca{T}_{\X}(\cdot)$ and $\NX(\cdot)$ denote the (regular) tangent cone and normal cone of $\X$ at a given point \cite[Theorem 6.9]{rockafellar2009variational}, respectively. Additionally, we denote $\E:= \mathrm{aff}(\X) - x$ for a certain $x \in \X$, but note that $\E$ is independent of the choice of $x \in \X$. 
                The dimension of a subspace $\ca{L} \subseteq \Rn$ is denoted as $\mathrm{dim}({\cal L})$.
    \begin{assumpt}
        \label{Assumption_f}
        \begin{enumerate}
            \item There exists an open set $\Y \subseteq \Rn$ such that $\X \subset \Y$ and 
            \begin{enumerate}
                \item The objective function $f: \Rn \to \bb{R}$ is continuously differentiable with locally Lipschitz continuous gradient over $\Rn$.  
                \item The constraint mapping $c: \Y \to \bb{R}^p$ is continuously differentiable with locally Lipschitz continuous Jacobian over $\Y$. 
            \end{enumerate}
            \item For any given $x \in \K$, there exists $r > 0$ (independent of $x$) and $\tau_x > 0$ 
            such that 
            \begin{enumerate}
                \item For any $y\in \{ y\in \Y \cap \mathrm{aff}(\X):\ \norm{y-x}\leq \tau_x \}$, it holds that $\mathrm{dim}(\{\nabla c(y)\tp d: d \in \E\})=r$. 
                \item For any $y \in \{y\in \X: \norm{y-x} \leq \tau_x\}$, it holds that 
                \begin{equation*}
                    \mathrm{dim}\Big( \left\{\Jc(y)\tp d: d \in \ca{T}_{\X}(y) \cap -\ca{T}_{\X}(y) \right\} \Big) = r. 
                \end{equation*}
            \end{enumerate}
        \end{enumerate}
    \end{assumpt}
    
    Assumption \ref{Assumption_f}(2) can be regarded as an extension of the relaxed constant rank constraint qualification (rCRCQ) \cite{solodov2010constraint}, in the sense that it coincides with rCRCQ when $\X = \Rn$. Additionally, Assumption \ref{Assumption_f}(2) reduces to the nondegeneracy condition \cite[subsection 4.6.1]{bonnans2013perturbation} when $r = p$. 
    
    Optimization problems that take the form of \eqref{Prob_Ori} have wide applications in various areas. Specifically, for optimization problems with constraints $u(x) \in \ca{Z}$, where $u: \Rq \to \Rm$ is continuously differentiable and $\Z \subseteq \Rm$ is a closed convex set, we can introduce a slack variable $y$ to reformulate the constraints as $u(x) - y = 0, ~y \in \Z$. Such a reformulation fits into the form of \eqref{Prob_Ori} with $c(x,y) = (u(x)-y, y)$ and $\X = \{0\} \times \Z$.

    For solving \eqref{Prob_Ori}, a wide range of optimization approaches are developed, including augmented Lagrangian methods \cite{curtis2015adaptive, jia2023augmented}, interior point methods \cite{yamashita1998globally, wright1998interior}, barrier methods \cite{nash1993barrier, luenberger2016penalty, polyak1992modified}, etc.  These methods handle the equality constraints $c(x) = 0$ by introducing a sequence of surrogate subproblems, leading to multiple-loop update schemes. Such multiple-loop update schemes result in trade-offs between achieving primal stationarity and reducing constraint violations, which potentially undermines their efficiency in practical implementations. Moreover, the theoretical convergence properties, including global convergence and worst-case complexity, are usually required to be established separately for each individual algorithm.
    Given the recent advances in efficient optimization methods and their corresponding theoretical analysis for nonlinear optimization problems over $\X$, it is natural for us to explore approaches that can directly utilize these methods for solving \eqref{Prob_Ori}, and this is the main goal of this paper. 

    \subsection{Penalty function approaches}

    Penalty function approaches have emerged as an important alternative for solving a constrained problem by replacing its constraints with a penalty term in the objective function. The penalty function approaches yield optimization problems that admit $x \in \X$ as their constraints, hence enjoying simpler constraints compared to the original problem \eqref{Prob_Ori}. 
    
    The quadratic penalty function, introduced by \cite{courant1943variational}, penalizes infeasibility by adding the term $\frac{\beta}{2}\norm{c(x)}^2$ to the objective function $f$, where $\beta > 0$ is the penalty parameter. However, the quadratic penalty function is usually inexact, hence requiring a sufficiently large penalty parameter $\beta$ for a solution with the desired feasibility. As shown in \cite{jorge2006numerical}, excessively large $\beta$ can lead to ill-conditioned penalty functions, making them challenging to be efficiently minimized over $\X$. 

    Another variant of the penalty function is the nonsmooth penalty function, which employs the nonsmooth term $\beta\norm{c(x)}$ to penalize infeasibility. While the nonsmooth penalty function could be exact for finite $\beta > 0$, the nonsmooth term $\beta\norm{c(x)}$ leads to a nonsmooth optimization problem over $\X$. Consequently, nonsmooth optimization methods \cite{davis2019stochastic,davis2020stochastic} are required, which often exhibit slow convergence in practice \cite{jorge2006numerical}. 

    The augmented Lagrangian method, proposed by \cite{hestenes1969multiplier, powell1969method}, incorporates dual variables $\lambda \in \Rp$ to form the augmented Lagrangian function $L_{\beta}(x, \lambda) := f(x) - \inner{\lambda, c(x)} + \frac{\beta}{2} \norm{c(x)}^2$. 
    Then 
    the constrained optimization problem \eqref{Prob_Ori} can be transferred to the minimax optimization problem $\min_{x \in \X} \max_{\lambda \in \Rp} L_{\beta}(x, \lambda)$ under mild conditions, as demonstrated in \cite[Theorem 11.59]{rockafellar2009variational}.
    As augmented Lagrangian methods are based on alternating updates for the minimax optimization problem, they typically involve a double-loop scheme. The dual variable $\lambda$ is updated in the outer loop using dual ascent methods, while the primal variable $x$ is updated in the inner loop by solving $\min_{x \in \X} L_{\beta}(x, \lambda_k)$ to a desired accuracy. Balancing the computational costs of these updates is challenging, and the efficiency of the methods depends on careful tuning of the penalty parameter and accuracy tolerances when solving the subproblems. 
    Even with careful fine-tuning, there is a limit to the efficiency of these augmented Lagrangian methods in practice \cite{gao2019parallelizable,xiao2022class}.
    
    Fletcher's penalty function, first proposed in \cite{fletcher1973exact}, is originally developed for optimization problems with equality constraints (i.e., $\X =\Rn$ in \eqref{Prob_Ori}). Based on the Lagrangian penalty function, the Fletcher's penalty function for \eqref{Prob_Ori} with $\X = \Rn$ can be expressed as 
    \begin{equation*}
        h_{\rm fl}(x) := f(x) - \inner{\lambda_{\rm fl}(x), c(x)} + \frac{\beta}{2} \norm{c(x)}^2,
    \end{equation*}
    where $\lambda_{\rm fl}(x) := \Jc(x)^{\dagger}\nabla f(x)$ acts as the dual variable, 
    $\Jc(x)$ is the transposed Jacobian of $c$, and $\Jc(x)^{\dagger}$ is the pseudo-inverse of $\Jc(x)$. Fletcher's penalty function has been extended to optimization over box constraints \cite{estrin2020implementing}, but its application to general $\X$ remains unexplored. Note that Fletcher's penalty function and its variants \cite{di1986exact, zavala2014scalable, 
    estrin2020implementing} require the first-order derivatives of $f$, and hence their differentiability depends on the second-order differentiability of $f$. Therefore, computing the gradients of Fletcher's penalty function involves the  
    second-order derivatives of $f$, which could be expensive and may not always be available in practice. As a result, existing methods based on Fletcher's penalty function, such as approximated steepest descent \cite{estrin2020implementing, goyens2024computing}, Newton methods \cite{toint1981towards, steihaug1983conjugate, zavala2014scalable}, and quasi-Newton methods \cite{estrin2020implementing}, 
    rely on approximation strategies for higher-order derivatives of $f$. These approximation techniques require significant modifications to the original algorithms, which reduce the compatibility of Fletcher's penalty function with existing optimization methods.

    Very recently, the constraint dissolving approach \cite{xiao2023dissolving,hu2022constraint,hu2022improved} is developed for solving an optimization problem with equality constraints (i.e., $\X =\Rn$ in \eqref{Prob_Ori}). The constraint dissolving function is an exact penalty function for \eqref{Prob_Ori} in the sense that it has the same first-order and second-order stationary points as \eqref{Prob_Ori} under mild conditions. Compared to Fletcher's penalty function, the constraint dissolving penalty function has the same order of smoothness as the objective function. Therefore, the exact gradients and Hessians of the constraint dissolving function are easier to compute, hence enabling the direct implementations of various existing unconstrained optimization algorithms for minimizing the constraint dissolving function. However, the existing constraint dissolving approaches are  developed only for equality-constrained settings with $\X = \Rn$ for \eqref{Prob_Ori}, which are not applicable 
    to the broader class of equality constrained optimization problems over a constraint set, particularly those problems involving box or conic constraints.

    \subsection{Constraint dissolving approach}
    To develop exact and efficient penalty function approaches for \eqref{Prob_Ori} with
    a general closed convex set $\X$, we consider the following constraint dissolving optimization problem,
    \begin{equation}
        \label{Prob_Pen}
        \tag{CDP}
        \min_{x \in \X} \quad h(x):= f(\A(x)) + \frac{\beta}{2} \norm{c(x)}^2. 
    \end{equation}
    Here $\beta \geq 0$ is the penalty parameter for \eqref{Prob_Pen}, while the mapping $\A: \X \to \Rn$ is the constraint dissolving mapping that satisfies the following regularity conditions. 
    \begin{assumpt}
        \label{Assumption_A}
        \begin{enumerate}
            \item The constraint dissolving mapping $\A$ is locally Lipschitz smooth over $\X$; 
            \item For any $x \in \K$, it holds that $\A(x) = x$ and $\Ja(x)\Jc(x) = 0$;
            \item There exists a locally Lipschitz continuous mapping $\RA: \X \to \bb{R}^{n\times n}$ such that 
            \begin{enumerate}
                \item $\RA(x) = 0$  for all $x \in \K$;
                \item for any $x \in \K$, there exists $\omega_x > 0$ such that
                    \begin{equation*}
                        \bigcup_{\norm{y- x}\leq \omega_x, ~y \in \X} \left(\Ja(y) - I_n- \RA(y)  \right) \NX(y) = \{0\}. 
                    \end{equation*}
                    In particular, $(\Ja(x)-I_n)\NX(x) = \{0\}$ holds for any $k \in \K$. 
            \end{enumerate}
        \end{enumerate}
    \end{assumpt}
    
    It is worth mentioning that the Lipschitz smoothness of $\A$ ensures the Lipschitz smoothness of $h$. As a result, compared to \eqref{Prob_Ori}, \eqref{Prob_Pen} dissolves the equality constraints $c(x) = 0$  while preserving the differentiability of the objective function. Moreover, as the constraint $x \in \X$ in \eqref{Prob_Pen} satisfies the nondegeneracy condition, \eqref{Prob_Pen} enjoys stronger constraint qualification than \eqref{Prob_Ori}. Additionally, Assumption \ref{Assumption_A}(3) states that for any $x$ in a neighborhood of $\K$, $\Ja(x)$ approximates the identity mapping within the subspace spanned by $\NX(x)$ with controlled approximation errors. Notably, when $\X = \Rn$, Assumption \ref{Assumption_A}(3) is automatically satisfied. 

    \subsection{Contributions}
    In this paper, we establish the equivalence between \eqref{Prob_Ori} and \eqref{Prob_Pen}. 
    In particular, for any $x \in \K$ and a sufficiently large but finite $\beta\geq 0$, we prove that any first-order stationary point of \eqref{Prob_Pen} in a neighborhood of $x$ is a first-order stationary point of \eqref{Prob_Ori}, and vice versa. Additionally, we show that any $\varepsilon$-first-order stationary point of \eqref{Prob_Pen} corresponds to a $2\varepsilon$-first-order stationary point of \eqref{Prob_Ori}. Moreover,  when $f$, $\A$, and $c$ are twice differentiable, within a neighborhood of $x\in\K$, \eqref{Prob_Ori} and \eqref{Prob_Pen} share the same second-order stationary points, second-order sufficient condition (SOSC) points, and strong SOSC points, respectively (see Section \ref{Subsection_Stationarity} for detailed definitions). Additionally, any point $x \in \K$ satisfying strict complementarity for \eqref{Prob_Ori} also satisfies strict complementarity for \eqref{Prob_Pen}. Under certain error bound conditions, we further demonstrate the global equivalence of \eqref{Prob_Ori} and \eqref{Prob_Pen} in the aspect of their first- and second-order stationary points. These equivalences between \eqref{Prob_Ori} and \eqref{Prob_Pen} demonstrate that \eqref{Prob_Pen} inherits various desirable theoretical properties from \eqref{Prob_Ori}. 

For the practical computation of \eqref{Prob_Pen}, it is important for us to construct constraint dissolving mappings $\A$ satisfying Assumption \ref{Assumption_A}. Thus,
we also investigate the construction of the constraint dissolving mapping $\A$ for \eqref{Prob_Ori}. A general scheme is proposed for constructing $\A$ based on the constraint mapping $c$ and its Jacobian, as well as a \textit{projective mapping} $Q: \X \to \bb{R}^{n \times n}$, which only depends on the structure of $\X$ (see Assumption \ref{Assumption_Q} for details). Explicit formulations of $Q$ are provided for a wide variety of constraint sets $\X$, ensuring that the resulting constraint dissolving mappings are computationally efficient. For certain specialized forms of $\K$, we introduce formulations of $\A$ that rely only on matrix-matrix multiplications for computing both $\A$ and its Jacobian $\Ja$. Therefore, existing optimization approaches designed for optimization over $\X$ can be directly applied to solve \eqref{Prob_Ori} through solving our proposed constraint dissolving problem \eqref{Prob_Pen}. Preliminary numerical experiments illustrate that applying existing solvers designed for optimization over $\X$ can achieve high computational efficiency to solve 
\eqref{Prob_Ori}
through our developed constraint dissolving approach.

    \subsection{Organization}
    The outline of the rest of this paper is as follows. In Section 2, we present the notations and preliminary concepts that are necessary for the proofs in this paper. We present the equivalence between \eqref{Prob_Ori} and \eqref{Prob_Pen} in Section 3. In Section 4, we discuss the implementation details of the proposed constraint dissolving approach, particularly on the construction of the constraint dissolving mapping. Preliminary numerical experiments are presented in Section 5 to demonstrate the efficiency of our proposed constraint dissolving approach. We conclude the paper in the last section.

    \section{Preliminaries}
    
    \subsection{Notation}
For any matrix $A\in\bb{R}^{n\times p}$, 
let $\mathrm{range}(A)$ be the subspace spanned by the column vectors of $A$, $\mathrm{null}(A)$ be the null space of $A$ (i.e., $\mathrm{null}(A) = \{d \in \Rp: A d = 0\}$), 
and $\norm{\cdot}$ denotes the $\ell_2$-norm of a vector or an operator. 
For a subset $\ca{C} \subseteq \Rn$, $\mathrm{range}(\ca{C})$ refers to the smallest subspace of $\Rn$ that contains $\ca{C}$, $\mathrm{lin}(\ca{C})$ refers to the largest subspace of $\Rn$ that is contained in $\ca{C}$, $\mathrm{aff}(\ca{C})$ refers to the affine hull of $\ca{C}$, and $\mathrm{ri}(\ca{C})$ refers to the relative interior of $\ca{C}$. Moreover, when $\ca{C}$ is a subspace of $\Rn$, $\ca{C}^{\perp}$ is defined as the largest subspace that is orthogonal to $\ca{C}$. Additionally, for any $w \in \Rn$, we use the notation $\inner{w,\ca{C}} := \{ \inner{w,d} : d\in \ca{C}\}$ and $w^{\perp} := \{w\}^{\perp}$.

The notation $\mathrm{diag}(A)$ and $\Diag(x)$
stand for the vector formed by the diagonal entries of a matrix $A$,
and the diagonal matrix with the entries of $x\in\bb{R}^n$ as its diagonal, respectively. 
We denote the $r$-th largest singular value of a matrix $A\in \bb{R}^{n\times p}$ by $\sigma_r(A)$, while $\sigma_{\min}(A)$ refers to the smallest singular value of 
$A$. Furthermore, 
the pseudo-inverse of $A$ is denoted by $A^\dagger \in \bb{R}^{p\times n}$, which satisfies $AA^\dagger A = A$, $A^\dagger AA^\dagger = A^\dagger$, and both $A^{\dagger} A$ and $A A^{\dagger}$ are symmetric \cite{golub2013matrix}.

For any closed subset $\ca{C} \subseteq \Rn$ and any $x \in \bb{R}^n$, we define the projection from $x \in \bb{R}^n$ to $\ca{C}$ as 
\begin{equation*}
	\Pi_{\ca{C}}(x) := \mathop{\arg\min}_{y \in \ca{C}} ~ \norm{x-y}. 
\end{equation*}  
Furthermore, $\mathrm{dist}(x, \ca{C})$ refers to the distance between $x$ and $\ca{C}$, i.e. $ \mathrm{dist}(x, \ca{C}) = \norm{x - \Pi_{\ca{C}}(x)}$.

The (transposed) Jacobian of the mapping $c$ and $\A$ is denoted as $\Jc(x) \in \bb{R}^{n\times p}$ and $\Ja(x) \in \bb{R}^{n\times n}$, respectively. Let $c_i$ and $\A_{i}$ be the $i$-th coordinate of the mapping $c$ and $\A$ respectively, then $\Jc$ and $\Ja$ can be expressed by
\begin{equation*}
	\Jc(x) := \Big[\nabla c_1(x), \ldots, \nabla c_p(x)\Big] \in \bb{R}^{n\times p},
	\quad \text{and} \quad   
	\Ja(x) := \Big[\nabla \A_1(x), \ldots, \nabla \A_n(x)\Big] \in \bb{R}^{n\times n}. 
\end{equation*}
Besides, $\DJa(x): d \mapsto \DJa(x)[d]$ denotes the second-order derivative of the mapping $\A$, which can be regarded as a linear mapping from $\bb{R}^n$ to $\bb{R}^{n\times n}$ that satisfies 
$$
\DJa(x)[d] = \sum_{i = 1}^n d_i \nabla^2 \A_i(x).
$$
Similarly,  $\DJc(x)$  denotes the second-order derivative of the mapping $c$, which satisfies $\DJc(x)[d] = \sum_{i = 1}^p d_i \nabla^2 c_i(x)$. Additionally, 
we use $\nabla f(\A(x))$ to denote $\nabla f(z)|_{z = \A(x)}$ in the 
rest of this paper. 

Furthermore, Assumption \ref{Assumption_f}(2) implies that there exists an open subset $\tilde{\Y} \subseteq \Y$ such that $\X \subset \tilde{\Y}$ and for any 
\begin{equation} \label{eq-M}
   x\in \M := \{y \in\tilde{\Y} \cap \mathrm{aff}(\X): c(y) = 0 \},
\end{equation}
it holds that the subspace $\Jc(x)\tp \E$ has constant dimension in a neighborhood of $x$. Then together with \cite[Proposition 3.3.4]{absil2008optimization}, we can conclude that $\M$ is an embedded submanifold of $\Rn$. For $x\in \M$, we denote $\ca{T}_{\M}(x) := \{d \in \Rn: d\tp \Jc(x) = 0 \} \cap \E = \mathrm{null}(\Jc(x)\tp) \cap \E$ as the tangent space of $\M$ at $x$, and $\ca{N}_{\M}(x) := \mathrm{range}(\Jc(x))+ \E^{\perp}$ as the normal space of $\M$ at $x$. It is a basic fact in linear algebra that $\mathrm{null}(\Jc(x)^\top)^\perp = \mathrm{range}(\Jc(x))$ and thus $\left( \mathrm{null}(\Jc(x)\tp) \cap \E \right)^{\perp} = \mathrm{range}(\Jc(x))+ \E^{\perp}$.

\subsection{Stationarity}
\label{Subsection_Stationarity}
In this subsection, we introduce the definitions of the stationary points for the constrained optimization problem \eqref{Prob_Ori} and the constraint dissolving problem \eqref{Prob_Pen}. 
We start with the following lemma illustrating the equivalence between $\mathrm{range}(\TX(x))$ and $\E$.
    \begin{lem}
        \label{Le_aux_define_E}
        For any $x \in \X$, it holds that $\mathrm{range}(\TX(x)) = \E$ and thus $\mathrm{lin}(\NX(x)) = \E^{\perp}$.
    \end{lem}
    \begin{proof}
        For any $x \in \X$, the convexity of $\X$ illustrates that $\X -x \subseteq \TX(x) \subseteq \mathrm{range}(\X - x)$. As a result, we can conclude that $\mathrm{range}(\TX(x)) = \mathrm{range}(\X - x) = \E$. This completes the proof. 
    \end{proof}

Let the orthogonal projection matrix to $\E$ be denoted as $P_{\E}$. We present the following auxiliary lemma on the uniqueness of the dual variable that corresponds to the constraint $x \in \X$. 
\begin{lem}
    \label{Le_aux_unique_Gamma}
    Suppose Assumption \ref{Assumption_f} holds. Then for any $x \in \K$, any $\lambda \in \Rp$, and any $\Gamma \in \mathrm{range}(\NX(x))$ that satisfy $\Jc(x) \lambda + \Gamma =0$, it holds that $\Gamma = 0$.
\end{lem}
\begin{proof}
    We prove this lemma by contradiction. That is, we assume that there exists $x \in \K$, $\lambda \in \Rp$ and $\Gamma \in \mathrm{range}(\NX(x)) \setminus \{0\}$ such that $\Jc(x) \lambda + \Gamma = 0$. Then it holds that $P_{\E}\Jc(x) \lambda + P_{\E}\Gamma = 0$.
    Since $P_{\E}\Jc(x) \lambda$ is unaffected by any changes in the
    component of $\lambda$ in $\mathrm{null}(P_{\E}\Jc(x))$,
    without loss of generality, we assume that $\lambda \in 
    \mathrm{null}(P_{\E}\Jc(x))^\perp = 
    \mathrm{range}(\Jc(x)^\top P_{\E}) \subseteq \Jc(x)\tp \E$. 
    Notice that $\Gamma \in \big(\mathrm{lin}(\TX(x))\big)^{\perp}$ and $\Gamma - P_{\E}\Gamma \in \E^{\perp} \subseteq \big(\mathrm{lin}(\TX(x))\big)^{\perp}$, we get that $-P_{\E}\Jc(x)\lambda = P_{\E}\Gamma \in \big(\mathrm{lin}(\TX(x))\big)^{\perp} $. Thus the equality $\inner{\lambda, \Jc(x)\tp  w} = \inner{\lambda, \Jc(x)\tp P_{\E} w}= \inner{P_\E\Jc(x) \lambda, w} = -\inner{P_\E \Gamma, w} = 0$ holds for any $w \in \mathrm{lin}(\TX(x))\subset \E$. 
    As a result, we can conclude that $\lambda \in \left(\Jc(x)\tp \mathrm{lin}(\TX(x))\right)^{\perp} \setminus \{0\}$. 
    Therefore, together with Assumption \ref{Assumption_f}(2(a)), it holds that 
    \begin{equation*}
        \mathrm{dim}\left(\Jc(x)\tp \mathrm{lin}(\TX(x)) \right) < \mathrm{dim}\left(\Jc(x)\tp \E \right) = r,
    \end{equation*}
    which contradicts to Assumption \ref{Assumption_f}(2(b)). Therefore, we can conclude that for any $x \in \K$, any $\lambda \in \Rp$ and any $\Gamma \in \mathrm{range}(\NX(x))$ that satisfy $\Jc(x) \lambda + \Gamma = 0$, it holds that $\Gamma = 0$. This completes the proof. 
\end{proof}

Next we present the following lemma that gives an explicit expression for $\mathrm{lin}(\ca{N}_{\K}(x))$.
\begin{lem}
    \label{Le_aux_linNK}
    Suppose Assumption \ref{Assumption_f} holds. Then for any $x \in \K$, it holds that $\mathrm{lin}(\ca{N}_{\K}(x)) = \mathrm{range}(\Jc(x)) + \E^{\perp}$, and thus $\mathrm{range}(\ca{T}_{\K}(x)) = \mathrm{null}(\Jc(x)\tp) \cap \E$. 
\end{lem}
\begin{proof}
    For any $x \in \K$, it follows from Lemma \ref{Le_aux_unique_Gamma} and \cite[Theorem 6.42]{rockafellar2009variational} that 
    \begin{equation*}
        \ca{N}_{\K}(x) = \ca{N}_{\M}(x) + \ca{N}_{\X}(x) = \mathrm{range}(\Jc(x))+ \E^\perp + \NX(x)=\mathrm{range}(\Jc(x)) + \NX(x).
    \end{equation*}
    Then for any $w \in \mathrm{lin}(\ca{N}_{\K}(x))$, there exists $\lambda_1, \lambda_2 \in \Rp$ and $\Gamma_1, \Gamma_2 \in \NX(x)$ such that $\Jc(x)\lambda_1 + \Gamma_1 = w $ and
    \begin{equation*}
        (\Jc(x)\lambda_1 + \Gamma_1)  + (\Jc(x)\lambda_2 + \Gamma_2) =0.
    \end{equation*}
    Therefore, it follows from Lemma \ref{Le_aux_unique_Gamma} that $\Gamma_1 + \Gamma_2 = 0$, and hence $\Gamma_1 \in 
    \mathrm{lin}(\NX(x)) = \E^{\perp}$. Therefore, we can conclude that  $w \in \mathrm{range}(\Jc(x)) + \E^{\perp}$ and thus $\mathrm{lin}(\ca{N}_{\K}(x)) \subseteq \mathrm{range}(\Jc(x)) + \E^{\perp}$.

    On the other hand, the inclusion $\mathrm{range}(\Jc(x)) + \E^{\perp} \subseteq \mathrm{range}(\Jc(x)) + \NX(x) 
    = \ca{N}_{\K}(x)$ leads to the fact that $\mathrm{range}(\Jc(x)) + \E^{\perp} \subseteq \mathrm{lin}(\ca{N}_{\K}(x))$. Therefore, we get $\mathrm{lin}(\ca{N}_{\K}(x)) 
    = \mathrm{range}(\Jc(x)) + \E^{\perp}$, and thus 
    \begin{equation*}
        \begin{aligned}
            &\mathrm{range}(\ca{T}_{\K}(x)) = \Big( \mathrm{lin}(\ca{N}_{\K}(x)) \Big)^{\perp} = \Big( \mathrm{range}(\Jc(x)) + \E^{\perp} \Big)^{\perp}\\
            ={}& \Big( \mathrm{range}(\Jc(x))\Big)^{\perp} \cap \Big(\E^{\perp} \Big)^{\perp} = \mathrm{null}(\Jc(x)\tp) \cap \E.
        \end{aligned}
    \end{equation*}
    This completes the proof. 
\end{proof}

For any $x \in \K$, it follows from Lemma \ref{Le_aux_unique_Gamma} and \cite[Theorem 6.42]{rockafellar2009variational} that  $\ca{N}_{\K}(x) = \ca{N}_{\M}(x) + \ca{N}_{\X}(x) = \mathrm{range}(\Jc(x)) + \NX(x)$. Therefore, we state the first-order optimality condition of \eqref{Prob_Ori} as follows.
\begin{defin}[\cite{clarke1990optimization}]\label{Defin_FOSP}
	Given $x \in \K$, we say that $x$ is a first-order stationary point of \eqref{Prob_Ori} if
	\begin{equation*}
            0 \in \nabla f(x) +  \mathrm{range}(\Jc(x)) + \NX(x).
	\end{equation*}
    Moreover, for any given $\varepsilon > 0$, we say that $x \in \X$ is an $\varepsilon$-first-order stationary point of \eqref{Prob_Ori} if 
    \begin{equation*}
        \mathrm{dist}\left( 0, \nabla f(x) + \mathrm{range}(\Jc(x)) + \NX(x) \right) \leq \varepsilon, \quad \text{and} \quad \norm{c(x)} \leq \varepsilon. 
    \end{equation*}
\end{defin}
Next we give the definitions of first-order stationary point and $\varepsilon$-first-order stationary point of \eqref{Prob_Pen}. 
\begin{defin}
    Given $x \in \X$, we say that $x$ is a first-order stationary point of \eqref{Prob_Pen} if
    \begin{equation*}
        0\in \nabla h(x) +\NX(x).
    \end{equation*}
    Moreover, for any given $\varepsilon > 0$, we say that $x$ is an $\varepsilon$-first-order stationary point of \eqref{Prob_Pen} if
    \begin{equation*}
        \mathrm{dist}\left( 0, \nabla h(x) +\NX(x) \right) \leq \varepsilon. 
    \end{equation*}
\end{defin}

We also introduce the definition of strict complementarity. 
\begin{defin}
    \label{Defin_strict_complementarity}
    Given $x \in \K$, we say that $x$ satisfies strict complementarity for \eqref{Prob_Ori} if $0 \in \nabla f(x) +  \mathrm{range}(\Jc(x)) + \mathrm{ri}(\NX(x))$. Moreover, given $x \in \X$, we say that $x$ satisfies strict complementarity for \eqref{Prob_Pen} if $0 \in \nabla h(x) + \mathrm{ri}(\NX(x))$. 
\end{defin}

In the following, we present the definitions of second-order tangent cones and support functions \cite{bonnans2013perturbation}, which are essential in characterizing the second-order optimality conditions for \eqref{Prob_Ori} and \eqref{Prob_Pen}. 
\begin{defin}
The following 
set is 
called the upper limit of a parameterized family  of subsets $A_t$ of $\Rn$:
\begin{equation*}
\limsup_{t \to t_0} A_t := \left\{x \in \Rn : \liminf_{t \to t_0} [\mathrm{dist}(x, A_t)] = 0\right\}.
\end{equation*}
\end{defin}

\begin{defin}
    For any closed subset $\ca{S}$ of $\Rn$, any $x \in \ca{S}$, and any $w \in \Rn$, the second-order tangent set of $\ca{S}$ to the point $x$ in direction $w$ is defined as 
    \begin{equation*}
        \ca{T}^2_{\ca{S}}(x, w) := \limsup_{t \to 0} \frac{\ca{S} - x - tw}{ \frac{1}{2}t^2}. 
    \end{equation*}
\end{defin}

From \cite[Definition 3.32]{bonnans2013perturbation}, we present the following definition on the second-order directional differentiability of a set. 
\begin{defin}
    \label{Defin_SO_directionally_diff}
    The set ${\cal S}$ is second-order directionally differentiable at a point $x\in{\cal S}$ in the direction $h\in{\cal T}_{\cal S}(x)$ if ${\cal T}^i_{\cal S}(x)={\cal T}_{\cal S}(x)$ and ${\cal T}^{i,2}_{\cal S}(x,h)={\cal T}^{2}_{\cal S}(x,h)$, where the definitions of ${\cal T}^i_{\cal S}(x)$ and ${\cal T}^{i,2}_{\cal S}(x,h)$ are given in \cite[Definition 2.54 and Definition 3.28]{bonnans2013perturbation}. 
\end{defin} 

\begin{defin}
    For any closed convex subset $\ca{S}$ of $\Rn$ and any $w \in \Rn$, the support function $\zeta$ is defined as 
    \begin{equation*}
        \zeta(w, \ca{S}) := \sup_{y \in \ca{S}} \inner{w, y}. 
    \end{equation*}
\end{defin}

Next, we introduce the concept of second-order stationary points for both \eqref{Prob_Ori} and \eqref{Prob_Pen}. 
For any $x \in \X$ that is a first-order stationary point of \eqref{Prob_Pen}, we define 
\begin{equation}\label{Eq_M1}
    \M_1(x) :=\{\Gamma \in \NX(x): \nabla h(x)+\Gamma=0 \} = \{-\nabla h(x)\},
\end{equation}
and 
\begin{equation*}
    \ca{C}_1(x) := \{d \in \Rn: d\in {\cal T}_{\cal X}(x), \inner{d, \nabla h(x)}=0\} = \{d \in \Rn: d\in \TX(x)\cap\Gamma^{\perp}, \forall ~\Gamma \in \M_1(x) \}.
\end{equation*}
Then we have the following definitions on the second-order stationarity of \eqref{Prob_Pen}, in the aspects of its second-order stationary points and second-order sufficient condition (SOSC). 

\begin{defin}
    \label{Defin_SONC_Pen}
    Suppose $x\in\X$ is a first-order stationary point of \eqref{Prob_Pen}. Then for any twice-differentiable $h$, let $\Gamma\in{\cal M}_1(x) = \{-\nabla h(x)\}$, we say that $x$ is a second-order stationary point of \eqref{Prob_Pen} if 
    \begin{equation}\label{Eq_SONC}
     \inner{d, \nabla^2 h(x)d} - \zeta(\Gamma, \ca{T}^2_{\X}(x, d)) \geq 0,\quad\forall\; 
     d\in{\cal C}_1(x).
    \end{equation}
\end{defin}

\begin{defin}
    \label{Defin_SOSC_Pen}
    Suppose $x\in\X$ is a first-order stationary point of \eqref{Prob_Pen}. Then for any twice-differentiable $h$, let $\Gamma \in \M_1(x) :=\{-\nabla h(x)\}$, we say that $x$ satisfies SOSC of \eqref{Prob_Pen} if  
    \begin{equation*}
     \inner{d, \nabla^2 h(x)d} - \zeta(\Gamma, \ca{T}^2_{\X}(x, d)) > 0, \quad \forall~ 
     0\neq d\in \ca{C}_1(x).
    \end{equation*}
     Moreover, we say that $x$ satisfies the strong SOSC of  \eqref{Prob_Pen} if 
     \begin{equation*}
     \inner{d, \nabla^2 h(x)d} - \zeta(\Gamma, \ca{T}^2_{\X}(x, d)) > 0, \quad\forall\;
    0\not= d\in {\rm range}\big(\TX(x)\cap\Gamma^{\perp}\big).
    \end{equation*}
\end{defin}

Next we present the concept of second-order stationarity for \eqref{Prob_Ori}. For any $x \in \K$ that is a first-order stationary point of \eqref{Prob_Ori}, we define
\begin{equation}\label{Eq_defC2}
    \begin{aligned}
        \M_2(x) :={}& \{(\lambda,\Gamma) \in \Rp \times \Rn: \nabla f(x)+\Jc(x)\lambda+\Gamma=0,\,  \Gamma \in \NX(x)\},
        \\
        \ca{C}_2(x):={}&\{d \in \Rn: d\in {\cal T}_{\cal X}(x),\, \inner{\nabla f(x), d}=0, \,d\tp \Jc(x)=0\}\\
        ={}&\{d \in \Rn:  d\in \TX(x)\cap\Gamma^{\perp},\, d\tp \Jc(x)=0, ~\forall (\lambda, \Gamma) \in \M_2(x)\}.
    \end{aligned}
\end{equation}
From Lemma \ref{Le_aux_unique_Gamma}, we can conclude that for any given $x\in \K$ that is a first-order stationary point of \eqref{Prob_Ori}, the choice of $\Gamma$ in $\M_2(x)$ is unique.

Then we present the following definitions on the second-order stationarity of \eqref{Prob_Ori}.
\begin{defin}
    \label{Defin_SONC_Ori}
    Suppose $x\in\K$ is a first-order stationary point of \eqref{Prob_Ori}. Then for any twice-differentiable $f$ and $c$, we say that $x$ is a second-order stationary point of \eqref{Prob_Ori} if  
    \begin{equation*}
     \sup_{(\lambda,\Gamma)\in{\cal M}_2(x)}\inner{d, \left(\nabla^2 f(x) + \nabla^2 c(x)[\lambda]\right)d} - \zeta(\Gamma, \ca{T}^2_{\X}(x, d)) \geq 0,
     \quad \forall\; d\in{\cal C}_2(x).
    \end{equation*}

\end{defin}

\begin{defin}
    \label{Defin_SOSC_Ori}
    Suppose $x\in\K$ is a first-order stationary point of \eqref{Prob_Ori}. Then for any twice-differentiable $f$ and $c$, we say that $x$ satisfies SOSC of \eqref{Prob_Ori} (or $x$ is called a SOSC point of \eqref{Prob_Ori}) if 
    \begin{equation*}
     \sup_{(\lambda,\Gamma)\in{\cal M}_2(x)}\inner{d, \left(\nabla^2 f(x) + \nabla^2 c(x)[\lambda]\right)d} - \zeta(\Gamma, \ca{T}^2_{\X}(x, d)) > 0, \quad \forall\; 0\not=d\in{\cal C}_2(x).
    \end{equation*}  
    Moreover, we say that $x$ satisfies the strong SOSC of \eqref{Prob_Ori} (or $x$ is called a strong SOSC point of \eqref{Prob_Ori}) if 
    $\forall~ 
     0\not=d\in \cap_{(\lambda, \Gamma) \in \M_2(x)} \{d\in {\rm range}\big(\TX(x)\cap\Gamma^{\perp}\big) : d\tp \Jc(x)=0 \}$, 
     it holds that
    \begin{equation*}
     \sup_{(\lambda,\Gamma)\in{\cal M}_2(x)}\inner{d, \left(\nabla^2 f(x) + \nabla^2 c(x)[\lambda]\right)d} - \zeta(\Gamma, \ca{T}^2_{\X}(x, d)) > 0.
    \end{equation*}
\end{defin}

    Finally, we present the following lemma to characterize the differentials of the constraint dissolving function $h$. These results directly follow from the expression of $h$ in \eqref{Prob_Pen} and the chain rule for Clarke subdifferentially regular mappings in \cite[Proposition 2.3.3, Corollary 1 and Theorem 2.3.10]{clarke1990optimization}. Therefore, we omit the proof of the following lemma for simplicity. 
    \begin{lem}
        \label{Le_Subdifferential_h}
        Suppose Assumption \ref{Assumption_f} and Assumption \ref{Assumption_A} hold. For any $x \in \X$, it holds that
        \begin{equation*}
            \nabla h(x) = \Ja(x) \nabla f(\A(x)) + \beta \Jc(x) c(x).
        \end{equation*}
        Furthermore, when $f$, $\A$ and $c$ are twice-differentiable, it holds that 
        \begin{equation*}
            \nabla^2 h(x) = \Ja(x) \nabla^2 f(\A(x)) \Ja(x)\tp + 
            \DJa(x)[\nabla f(\A(x))] + \beta (\DJc(x) [c(x)] + \Jc(x) \Jc(x)\tp). 
        \end{equation*}
    \end{lem}

\subsection{Comments on constraint qualifications}

In this subsection, we give some comments on the constraint qualification of \eqref{Prob_Ori}. It is worth mentioning that Assumption \ref{Assumption_f}(2) reduces to the nondegeneracy condition when $r = p$, hence it is weaker than the requirements in \cite[Assumption A2]{estrin2020implementing}. Moreover, we present the following example to demonstrate that Assumption \ref{Assumption_f}(2) is no stronger than the Robinson constraint qualification \cite{robinson1976regularity}. 
    \begin{rmk}
        \label{Rmk_CQs}
        Consider the following constrained optimization problem,
        \begin{equation}
            \label{Eq_Rmk_CQs}
            \begin{aligned}
                \min_{x_1 \geq 0, ~x_2 \in \bb{R}} \quad &f(x_1, x_2)\\
                \text{s. t.} \quad & x_1 + x_2 = 0, \quad 2x_1 + 2x_2 = 0. 
            \end{aligned}
        \end{equation}
        As the constraints in \eqref{Eq_Rmk_CQs} are redundant,  \eqref{Eq_Rmk_CQs} does not satisfy the Robinson constraint qualification. 
        
        Conversely, for any $y \in \bb{R}_+ \times \bb{R}$, it holds that 
        \begin{equation*}
            \mathrm{dim}\left(\left\{ 
                \Jc(y)\tp d\,:\,
                d \in \ca{T}_{\X}(y) \cap -\ca{T}_{\X}(y) \right\}\right) = 3.
        \end{equation*}
        Therefore,  \eqref{Eq_Rmk_CQs} satisfies Assumption \ref{Assumption_f}(2), which illustrates that Assumption \ref{Assumption_f}(2) is no stronger than the Robinson constraint qualification. It is worth mentioning that when \eqref{Prob_Ori} is reduced to the cases where $\X = \Rn$, Assumption \ref{Assumption_f}(2) is neither weaker nor stronger than the Robinson constraint qualification \cite[p.5]{solodov2010constraint}. 
    \end{rmk}

    In the rest of this subsection, we aim to show that Assumption \ref{Assumption_f}(2) can be satisfied under mild conditions. For any $u_1 \in \Rn$ and $u_2 \in \Rp$, let the perturbed feasible set be defined by 
    \begin{equation}
        \K_{(u_1, u_2)} := \{x - u_1 \in \X: c(x) - u_2 = 0\}.
    \end{equation}
    Then from the results in \cite{drusvyatskiy2016generic,tang2024feasible}, it holds that Assumption \ref{Assumption_f}(2) can be generically guaranteed for $\K_{(u_1, u_2)}$ with any semi-algebraic \cite{lojasiewicz1965ensembles} constraint mapping $c$ and any semi-algebraic set $\X$, as demonstrated in the following proposition.
    \begin{prop}(\cite[Theorem 5.2]{drusvyatskiy2016generic})
        Suppose Assumption \ref{Assumption_f}(1) holds, the mapping $c$ and the set $\X$ are semi-algebraic, and there exists $\delta > 0$ such that the feasible set $\K_{(u_1, u_2)} := \{x - u_1 \in \X: c(x) - u_2 = 0\}$ is non-empty for all $(u_1, u_2) \in \ca{B}_{\delta}$, where $\ca{B}_{\delta}$ denotes the ball centered at $0$ with radius $\delta$ . Then for almost every $(u_1, u_2) \in \ca{B}_{\delta}$, Assumption \ref{Assumption_f}(2) holds for $\K_{(u_1, u_2)}$. 
    \end{prop}

    \section{Equivalence}

    In this section, we demonstrate the equivalence between problems \eqref{Prob_Ori} and \eqref{Prob_Pen}. Specifically, within a neighborhood of $\X$, we show that \eqref{Prob_Ori} and \eqref{Prob_Pen} share the same first-order stationary points, second-order stationary points, SOSC points, and strong SOSC points. Furthermore, when the error bound condition in Assumption \ref{Assumption_error_bound} holds, we establish the global equivalence between \eqref{Prob_Ori} and \eqref{Prob_Pen}, in the sense that they have the same first-order stationary points over $\X$.

    This section is structured as follows. Section \ref{Subsection_31} introduces preliminary lemmas characterizing the properties of the constraint dissolving mapping $\A$. Section \ref{Subsection_32} establishes the equivalence between \eqref{Prob_Ori} and \eqref{Prob_Pen} in a neighborhood of $\X$. Section \ref{Subsection_33} shows the global equivalence between \eqref{Prob_Ori} and \eqref{Prob_Pen} under the error bound condition and compactness of $\X$ in Assumption \ref{Assumption_error_bound}.

    \subsection{Basic properties of constraint dissolving mappings}
    \label{Subsection_31}

    In this subsection, we present some basic properties of the constraint dissolving mapping $\A$ and its corresponding Jacobian $\Ja(x)$. We begin our analysis with the following auxiliary lemma.
    \begin{lem}\label{Le_SONC_aux0}
        Suppose Assumption \ref{Assumption_f} holds. Then for any given function $\phi: \X \to \bb{R}$ that is locally Lipschitz smooth over $\X$ and satisfies $\phi(y) = 0$ for any $y \in \K$, it holds for any $x \in \K$ that $\nabla \phi(x) \in \mathrm{range}(\Jc(x)) + \E^{\perp}$. 
    \end{lem}
    \begin{proof}
        For any $x \in \K$ and any  $d \in \ca{T}_{\K}(x)$, there exists a sequence of points $\{\xk\} \subset \K$ and a sequence of positive  numbers $\{t_k\}$ such that $\lim_{k\to +\infty} t_k = 0$, $\lim_{k\to +\infty} \xk = x$, and $\lim_{k\to +\infty} \frac{\xk - x}{t_k} = d$. Since $\phi(\xk) = 0$ holds for any $k\geq 0$, then it holds that 
        \begin{equation*}
            \inner{d, \nabla \phi(x)} = \lim_{k\to +\infty} \frac{\phi(\xk) - \phi(x)}{t_k} = 0.
        \end{equation*}
        Together with Lemma \ref{Le_aux_linNK}, it holds that $\nabla \phi(x) \in \mathrm{lin}(\ca{N}_{\K}(x)) = \mathrm{range}(\Jc(x)) + \mathrm{lin}(\ca{N}_{\X}(x)) = \mathrm{range}(\Jc(x)) + \E^{\perp}$. This completes the proof. 
    \end{proof}

    The following lemma 
    gives some properties of the range space of $\Ja(x)\tp$ for any $x \in \K$. 
    \begin{lem}
        \label{Le_range_JAtp}
        Suppose Assumption \ref{Assumption_f} and Assumption \ref{Assumption_A} hold. For any given $x \in \K$, the inclusion $\Ja(x)\tp d \in \mathrm{null}(\Jc(x)\tp)$ holds for all $d \in \Rn$. Moreover, for any $d \in \mathrm{null}(\Jc(x)\tp)\cap \E$, it holds that $\Ja(x)\tp d = d$.
    \end{lem}
    \begin{proof}
        First, for any $x \in \K$, any $d \in \Rn$ and any $d_1 \in \mathrm{range}(\Jc(x))$, from the fact that $\Ja(x)\Jc(x) = 0$, we can conclude that $\inner{d_1, \Ja(x)\tp d} = 0$ holds. Then from the arbitrariness of $d_1\in \mathrm{range}(\Jc(x))$, we can conclude that 
        $\Ja(x)\tp d \in  \mathrm{range}(\Jc(x))^\perp = \mathrm{null}(\Jc(x)\tp)$ holds for all $d \in \Rn$.

        On the other hand, for any $w \in \Rn$, consider the auxiliary function $\phi(y) = \inner{w, \A(y) - y}$ for $y \in \K$. Then from Lemma \ref{Le_SONC_aux0}, for any $x \in \K$ and any $d \in \mathrm{null}(\Jc(x)\tp) \cap \E$, it holds that 
        \begin{equation*}
            0 = \inner{\nabla \phi(x), d} = \inner{\Ja(x) w - w, d} = \inner{w, (\Ja(x)\tp - I_n)d}.
        \end{equation*}
        Then from the arbitrariness of $w \in \Rn$, it holds that $\Ja(x)\tp d = d$ for any $d \in \mathrm{null}(\Jc(x)\tp) \cap \E$. This completes the proof. 
    \end{proof}

    Now for any $x \in \X$, we denote the projection matrix to $\mathrm{null}(\Jc(x)\tp)\cap \E$ as $P_{T}(x)$, and define $P_{N}(x):= I_n - P_{T}(x)$. From Assumption \ref{Assumption_f}, it holds that both $P_{T}$ and $P_{N}$ are locally Lipschitz continuous in a neighborhood of $\K$. The following lemma illustrates the relationship between $\Ja(x)$ and $P_{T}(x)$ for $x\in \K.$
    \begin{lem}
        \label{Le_JaPNc}
        Suppose Assumption \ref{Assumption_f} and Assumption \ref{Assumption_A} hold. For any $x \in \K$, it holds that $P_{T}(x) \Ja(x)  =P_{T}(x)$. 
    \end{lem}
    \begin{proof}
        For any $d \in \Rn$, it follows from Lemma \ref{Le_range_JAtp} that 
        \begin{equation*}
            \Ja(x)\tp (P_{T}(x) d ) =  P_{T}(x) d. 
        \end{equation*}
        Therefore, from the arbitrariness of $d$, we get $ \Ja(x)\tp P_{T}(x)  =P_{T}(x)$. This completes the proof. 
    \end{proof}

    The following lemma characterizes the null space of $\Ja(x)$ for any $x \in \K$. 
    \begin{lem}
        \label{Le_JC_Null}
        Suppose Assumption \ref{Assumption_f} and Assumption \ref{Assumption_A} hold. For any $x \in \K$, the inclusion $0 \in \Ja(x)d + \NX(x)$ holds if and only if $0 \in d+ \mathrm{range}(\Jc(x)) + \NX(x)$. 
        Moreover, the inclusion $0 \in \Ja(x)d + \mathrm{ri}(\NX(x))$ holds if and only if $0 \in d+ \mathrm{range}(\Jc(x)) + \mathrm{ri}(\NX(x))$.
    \end{lem}
    \begin{proof}
        We first prove the ``if'' part of this lemma. 
        For any $x \in \K$ and $d \in \mathrm{range}(\Jc(x)) -\NX(x)$, there exists $d_1, d_2 \in \Rn$ such that   $d_1 \in \mathrm{range}(\Jc(x))$, $d_2 \in \NX(x)$, and $d = d_1 - d_2$. From Assumption \ref{Assumption_A}(2)-(3), it holds that 
        \begin{equation*}
            \Ja(x) d_1 = 0, \quad \text{and} \quad \Ja(x) d_2 = d_2. 
        \end{equation*}
        Therefore, it holds that $\Ja(x) d = -d_2$, hence 
        \begin{equation*}
            0 \in -d_2 + \NX(x) = \Ja(x) d + \NX(x) . 
        \end{equation*}
        This completes the first part of the proof.

        Next we prove the ``only if'' part of this lemma. For any $d \in \Rn$ such that $0 \in \Ja(x)d + \NX(x)$, we have 
        $d-\Ja(x)d\in d+\NX(x)$. It holds for any $d_3 \in \mathrm{null}(\Jc(x)\tp)\cap \E$,
        $$\inner{d_3,d-\Ja(x)d}=\inner{d_3,d}-\inner{d_3,\Ja(x)d}=\inner{d_3,d}-\inner{\Ja(x)^{\top}d_3,d}=0,$$
        where the last equality follows from Lemma \ref{Le_range_JAtp}. From the arbitrariness of $d_3$, we know that
        \begin{equation*}
            d-\Ja(x)d\in \Big(\mathrm{null}(\Jc(x)\tp) \cap \E \Big) ^\perp = 
        \mathrm{range}(\Jc(x)) + \E^{\perp}.
        \end{equation*}
        It follows that $0\in d+\NX(x)-(d-\Ja(x)d)\subseteq d+\NX(x)+\mathrm{range}(\Jc(x))$. 
        This completes the proof of the first statement of the lemma.

        Furthermore, for any $x \in \X$ and any $w \in \mathrm{ri}(\NX(x))$, there exists $\hat{\delta}_w > 0$ such that $\{y \in \mathrm{aff}(\NX(x)) : \norm{y-w}\leq \hat{\delta}_w\} \subseteq \NX(x)$. Then for any $z \in \mathrm{lin}(\NX(x))$, it holds that 
        \begin{equation*}
            z + \{y \in \mathrm{aff}(\NX(x)) : \norm{y-w}\leq \hat{\delta}_w\} \subseteq \mathrm{lin}(\NX(x)) + \NX(x) \subseteq \NX(x),
        \end{equation*}
        which further implies $z + w \in \mathrm{ri}(\NX(x))$. Therefore, from the arbitrariness of $w \in \mathrm{ri}(\NX(x))$ and $z \in \mathrm{lin}(\NX(x))$, we can conclude that $\mathrm{ri}(\NX(x)) + \mathrm{lin}(\NX(x)) = \mathrm{ri}(\NX(x))$.  Following the same technique, we can easily prove the second statement when 
        $\NX(x)$ is replaced by $\mathrm{ri}(\NX(x))$. This completes the proof.
    \end{proof}

    Recalling that $P_{\E}$ refers to the matrix for the orthogonal projection onto $\E$, we illustrate the idempotent property of $\Ja(x)$ for any $x \in \K$ in the following lemma. 
    \begin{lem}
        \label{Le_Ja_ideo}
        Suppose Assumption \ref{Assumption_f} and Assumption \ref{Assumption_A} hold. For any $x \in \K$, it holds that $P_{\E}\Ja(x)^2 = P_{\E}\Ja(x)$. 
    \end{lem}
    \begin{proof} 
        For any $x \in \K$, consider any $d_1 \in \mathrm{null}(\Jc(x)\tp)\cap \E$ and any $d_2 \in \mathrm{range}(\Jc(x)) + \E^{\perp}$, 
        we can decompose $\Ja(x) d_1 = d_3 + d_4$ with
        $d_3 \in \mathrm{null}(\Jc(x)\tp)\cap \E$ and $d_4 \in \mathrm{range}(\Jc(x)) + \E^{\perp}$.
        Using Lemma \ref{Le_range_JAtp},
        we know that $\inner{d_1,d_1} = \inner{\Ja(x)^\top d_1, d_1}=
        \inner{d_1,d_3+d_4}=\inner{d_1,d_3}$. Similarly,  we have that $\inner{d_3,d_1}=\inner{\Ja(x)^\top d_3, d_1}=\inner{d_3,\Ja(x) d_1}=\inner{d_3,d_3+d_4}=\inner{d_3,d_3}$. Thus we have $\norm{d_1-d_3}^2 = 0$ and hence
        $d_1 = d_3$. Noting that Assumption \ref{Assumption_A}(2)-(3) and Lemma \ref{Le_aux_define_E} illustrate that $\Ja(x) d_4 \in \E^{\perp}$, we have
        \begin{equation*}
            P_{\E}\Ja(x)^2 d_1 = P_{\E}\Ja(x)(d_3 + d_4) = P_{\E}\Ja(x) d_3 =  P_{\E}\Ja(x) d_1. 
        \end{equation*}

        Moreover, recall that $d_2 \in \mathrm{range}(\Jc(x)) + \E^{\perp}$. Let $w_1 \in \mathrm{range}(\Jc(x))$ and $w_2 \in \E^{\perp}$ such that $d_2 = w_1 + w_2$. Then Assumption \ref{Assumption_A}(2)-(3) imply that $\Ja(x)d_2 = \Ja(x)w_1 + \Ja(x) w_2 = w_2.$ Thus        
        $P_{\E}\Ja(x)d_2 = P_{\E} w_2 = 0$ and $P_{\E}\Ja(x)^2d_2 = P_{\E}\Ja(x) w_2 = P_{\E} w_2 = 0$. Therefore, it holds that 
        \begin{equation*}
            P_{\E}\Ja(x)^2(d_1 + d_2) = P_{\E}\Ja(x) (d_1 + d_2). 
        \end{equation*}
        From the arbitrariness of $d_1 \in \mathrm{null}(\Jc(x)\tp)\cap \E$ and $d_2 \in \mathrm{range}(\Jc(x)) + \E^{\perp}$, we have that $P_{\E}\Ja(x)^2 = P_{\E}\Ja(x)$ holds for any $x \in \K$. This completes the proof. 
    \end{proof}

    \subsection{Equivalence on stationary points}
    \label{Subsection_32}
    In this subsection, we present the equivalence between \eqref{Prob_Ori} and \eqref{Prob_Pen} in the aspect of their first-order stationary points, second-order stationary points, SOSC points, and strong SOSC points. We first introduce some basic notation and constants in Section \ref{Subsection_constants}. Then, in Section \ref{Subsection_local_equivalence}, we show that \eqref{Prob_Ori} and \eqref{Prob_Pen} have the same first-order stationary points in a neighborhood of the given point in the feasible region $\K$, for sufficiently large but finite penalty parameter $\beta$. Moreover, in Section \ref{Subsection_local_secondorder_equivalence}, we prove that \eqref{Prob_Ori} and \eqref{Prob_Pen} have the same second-order stationary points, SOSC points, and strong SOSC points in a neighborhood of the given point in the feasible region $\K$. 
    

    \subsubsection{Constants}
    \label{Subsection_constants}

    In this part, we first introduce some necessary constants in our theoretical analysis. For any $x \in \X$, we define 
    \begin{equation*}
        \pi(x) := \sigma_{r}(P_{\E}\Jc(x)), 
    \end{equation*}
    where $\sigma_r(\cdot)$ denotes the $r-$th largest singular value of a matrix, as defined in Section 2.1. 
    Under Assumption \ref{Assumption_f}(2), $\pi(x) > 0$ holds for any $x \in \K$. 
    Then based on \cite[Lemma 1]{xiao2022cdopt}, we have the following lemma illustrating the relationship between $\norm{c(y)}$ and $\mathrm{dist}(y, \M)$, and $\norm{P_{\E}\Jc(y)c(y)}$.
    \begin{lem}
        \label{Le_aux_relationship_c_dist_PEJcc}
        Suppose Assumption \ref{Assumption_f} holds. Then for any $x \in \K$, there exists $\kappa_x > 0$ such that for any $y \in \ca{B}_{\kappa_x}(x)\cap \mathrm{aff}(\X)$, it holds that 
        \begin{equation*}
            \mathrm{dist}(y, \M) \leq \frac{2}{\pi(x) } \norm{c(y)}, \quad \text{and} \quad  \norm{P_{\E}\Jc(y)c(y)} \geq \frac{\pi(x)}{2} \norm{c(y)}.
        \end{equation*}
    \end{lem}
    \begin{proof}
        Let $x_0 = \mathop{\arg\min}_{z \in \mathrm{aff}(\X)} \norm{z}$. Note that from the optimality condition of $x_0$, we have that $P_\E x_0 = 0$.
        Consider the auxiliary mapping $c^{\sharp}(z) := c(P_{\E}z + x_0)$ for $z\in \mathbb{R}^n$. Then it is easy to verify that $c^{\sharp}(z) = c(z)$ and $\nabla c^{\sharp}(z) = P_{\E}\Jc(z)$ holds for any $z \in \mathrm{aff}(\X)$. Moreover, from Assumption \ref{Assumption_f}(2), for any $x \in \X$ and any $y \in \{z\in \Rn: \norm{z-x}\leq \tau_x, P_{\E}z + x_0 \in \tilde{\Y}\}$, it holds that
        \begin{equation*}
            \mathrm{dim}(\{\nabla c^{\sharp}(y)\tp d: d \in \Rn\}) = \mathrm{dim}(\{\nabla c(P_{\E}y + x_0)\tp P_{\E} d: d \in \Rn\}) = r.
        \end{equation*}
        Therefore, rCRCQ holds over the subset $\M^{\sharp}:= \{z \in \Rn: c^{\sharp}(z) = 0, P_{\E}z + x_0 \in \tilde{\Y}\}$, 
        where $\tilde{\Y}$ is given in the definition of $\M$ in \eqref{eq-M}. It is not difficult to show  that
         $\M^{\sharp} = \M + \E^{\perp}$. As a result, for any $y \in \mathrm{aff}(\X)$, it holds that $\mathrm{dist}(y, \M^{\sharp}) = \mathrm{dist}(y, \M)$. Additionally, from the definition of $\pi(x)$, for any $x \in \K$, it holds that $\sigma_r(\nabla c^{\sharp}(x)) = \pi(x)$. 
        
        Then by \cite[Lemma 1, Lemma 5]{xiao2022cdopt}, for any $x \in \K$, there exists $\kappa_x > 0$ such that $\mathrm{dist}(y, \M^{\sharp}) \leq \frac{2}{\pi(x) } \norm{c^{\sharp}(y)}$ and $\norm{\nabla c^{\sharp}(y) c^{\sharp}(y)} \geq \frac{\pi(x)}{2} \norm{c^{\sharp}(y)}$ hold for any $y \in \ca{B}_{\kappa_x}(x)\cap \mathrm{aff}(\X) \subseteq \ca{B}_{\kappa_x}(x)$. 
        Together with the facts that $\mathrm{dist}(y, \M^{\sharp}) = \mathrm{dist}(y, \M)$, $\nabla c^{\sharp}(y) = P_{\E}\Jc(y)$ and $c^{\sharp}(y) = c(y)$, we get $\mathrm{dist}(y, \M) \leq \frac{2}{\pi(x) } \norm{c(y)}$ and $\norm{P_{\E}\Jc(y)c(y)} \geq \frac{\pi(x)}{2} \norm{c(y)}$. This completes the proof. 
    \end{proof}

    Then from the local boundedness of $\pi$, we define $\rho_x$  for $x\in\K$ as 
    \begin{equation*}
    	\rho_x := \mathop{\arg\max}_{0 < \rho \leq \min\{\tau_x, \omega_x, \kappa_x\}}~ \rho \quad \text{s.t.} ~ \inf \left\{\pi(y): y \in \X, \norm{y-x} \leq \rho \right\} \geq \frac{1}{2} \pi(x), 
    \end{equation*}
    where $\tau_x$ is given in Assumption \ref{Assumption_f}(2).
    Here the subscript of $\rho_x$ emphasizes its dependence on the choices of $x$, and $\omega_x$ is defined in Assumption \ref{Assumption_A}(3). Based on the definition of $\rho_x$, we can define the set $\Theta_x := \{y\in \X: \norm{y-x}\leq \rho_x\}$ and define several constants in Table \ref{Table_Constants}.

    \begin{table}[tb]
\centering
\begin{tabular}{l|l|l}
\hline
\textbf{Constants} & \textbf{Definition} & \textbf{Description} \\ \hline
$\sigmaxc$ & $\pi(x)$ & The value of $\pi$ at $x$. \\ \hline
$\Mxf$ & $\sup_{y \in \Theta_x}~ \norm{\nabla f (\A(y))}$ & Upper bound of $\norm{\nabla f(\A(y))}$ over $\Theta_x$. \\ \hline
$\Mxc$ & $\sup_{y \in \Theta_x}   \norm{\Jc(y)}$ & Upper bound of $\norm{\Jc(y)}$ over $\Theta_x$. \\ \hline
$\Mxa$ & $\sup_{y \in \Theta_x}  \norm{\Ja(y)}$ & Upper bound of $\norm{\Ja(y)}$ over $\Theta_x$. \\ \hline
$\Mxr$ & $\sup_{y \in \Theta_x}  \frac{\norm{\RA(y)}}{\norm{c(y)}}$ & Upper bound of $\frac{\norm{\RA(y)}}{\norm{c(y)}}$ over $\Theta_x$. 
\\ \hline
$\Mxi$ & $\sup_{y \in \Theta_x}  \frac{\norm{P_{\E}\Ja(y)(\Ja(y) - I_n)}}{\norm{c(y)}}$ & Upper bound of $\frac{\norm{P_\E\Ja(y)(\Ja(y) - I_n)}}{\norm{c(y)}}$ over $\Theta_x$. 
\\ \hline
$\Mxjc$ & $\sup_{y \in \Theta_x} \frac{\norm{P_T(y)(\Ja(y) - I_n)}}{\norm{c(y)}}$ & Upper bound of $\frac{\norm{P_T(y)(\Ja(y) - I_n)}}{\norm{c(y)}}$ over $\Theta_x$. 
\\ \hline
$\Mxb$ & $\sup_{y \in \Theta_x} \frac{\norm{\Ja(y)\Jc(y)}}{\norm{c(y)}}$ & Upper bound of $\frac{\norm{\Ja(y)\Jc(y)}}{\norm{c(y)}}$ over $\Theta_x$. 
\\ \hline
$\Lxf$ & $\sup_{y, z\in\{\A(z): z \in \Theta_x\}, y\neq z}\frac{\norm{\nabla f(y) - \nabla f(z)}}{\norm{y-z}}$ & Lipschitz constant of $\nabla f$ over $\{\A(z): z \in \Theta_x\}$. \\ \hline
$\Lxa$ & $\sup_{y, z\in\Theta_x, y\neq z}\frac{\norm{\Ja(y) -  \Ja(z)}}{\norm{y-z}}$ & Lipschitz constant of $\Ja$ over  $\Theta_x$. \\ \hline
\end{tabular}
\caption{Definitions and descriptions of constants associated with $\Theta_x$. Here $P_T(y)$ is defined as the projection onto $\mathrm{null}(\Jc(y)\tp)\cap\E$.
}
\label{Table_Constants}
\end{table}

    \begin{rmk}
        Notice that for any $y \in \Theta_x \subseteq \ca{B}_{\kappa_x}\cap  \mathrm{aff}(\X)$, Lemma \ref{Le_aux_relationship_c_dist_PEJcc} illustrates that 
        \begin{equation*}
		\mathrm{dist}(y, \M) \leq \frac{2}{\sigmaxc } \norm{c(y)}.
	\end{equation*}
        Therefore, from the fact that $R_{\A}(x) =0$, $\Ja(x)\Jc(x) = 0$, $P_{\E}\Ja(x)(\Ja(x) - I_n) = 0$ (by Lemma \ref{Le_Ja_ideo}) and 
        $P_T(x)(\Ja(x) - I_n)= 0$ (by Lemma \ref{Le_JaPNc}) for any $x \in \K$, together with the Lipschitz continuity of $R_{\A}$, $\Ja$ and $P_T(x)$, we can conclude that the constants $\Mxr$, $\Mxb$, $\Mxi$ and $\Mxjc$ in Table \ref{Table_Constants} are well defined and finite for any $x \in \K$. 
    \end{rmk}

    Based on these constants, we further set 
    \begin{equation}\label{eq:deltax}
        \delta_x := \min\left\{ \rho_x, 
        \frac{\sigmaxc}{8(\Mxr\Mxc+ \Mxb)\Mxc},
        \frac{1}{\Mxc}  \right\},
    \end{equation}
    and define 
    \begin{equation}
        \Omega_x := \left\{ y \in \X: \norm{y-x}\leq \delta_x   \right\}. 
    \end{equation}
    Moreover, we define the threshold value of $\beta$ as 
    \begin{equation}\label{Eq_betax_defin}
        \begin{aligned}
            \beta_x := \max &\left\{ \frac{8\Mxf  (\Mxi + \Mxr\Mxa)}{\sigmaxc}, 
            \frac{\Mxf\Lxa (1+ \Mxa +M_{x,A}^2) }{\sigmaxc^2} \right\}.
        \end{aligned}
    \end{equation}
    
    \subsubsection{Equivalence on first-order stationary points}
    \label{Subsection_local_equivalence}
    In this part, we illustrate the equivalence between \eqref{Prob_Ori} and \eqref{Prob_Pen} in a neighborhood of the feasible region $\K$. 

    We first present the following theorem showing that any first-order stationary point of \eqref{Prob_Pen} on $\K$ is a first-order stationary point of \eqref{Prob_Ori}. 
    \begin{theo}
        \label{Theo_equivalence_feabile}
        Suppose Assumption \ref{Assumption_f} and Assumption \ref{Assumption_A} holds. Then for any $x \in \K$, $x$ is a first-order stationary point of \eqref{Prob_Pen} if and only if $x$ is a first-order stationary point of \eqref{Prob_Ori}. 
    \end{theo}
    \begin{proof}
        For any $x \in \K$ that is a first-order stationary point of \eqref{Prob_Pen}, it holds that 
        \begin{equation*}
            0 \in \nabla h(x) +  \NX(x) =  \Ja(x) \nabla f(\A(x)) + \NX(x) = \Ja(x) \nabla f(x) + \NX(x). 
        \end{equation*}
        Therefore, we can conclude from Lemma \ref{Le_JC_Null} that $0 \in \nabla f(x) + \NX(x) + \mathrm{range}(\Jc(x))$, hence $x$ is a first-order stationary point of \eqref{Prob_Ori}. 

        On the other hand, if $x \in \K$ is a first-order stationary point of \eqref{Prob_Ori}, Lemma \ref{Le_JC_Null} implies that 
        \begin{equation*}
            0 \in \Ja(x) \nabla f(x) + \NX(x) = \Ja(x) \nabla f(\A(x)) + \NX(x)  = \nabla h(x) + \NX(x). 
        \end{equation*}
        Then we can conclude that $x$ is a first-order stationary point of \eqref{Prob_Pen}. This completes the proof. 
    \end{proof}

    \begin{prop}
        Suppose Assumption \ref{Assumption_f} and Assumption \ref{Assumption_A} hold. For any $x \in \K$ that is an $\varepsilon$-first-order stationary point of \eqref{Prob_Ori},  $x$ is an $\varepsilon\Mxa$-stationary point of \eqref{Prob_Pen}. 
    \end{prop}
    \begin{proof}
        For any $x \in \K$ that is an $\varepsilon$-first-order stationary point of \eqref{Prob_Ori}, there exists $\lambda \in \Rp$ and $\Gamma \in \NX(x)$ such that 
        \begin{equation*}
            \norm{\nabla f(x) + \Jc(x)\lambda + \Gamma} \leq \varepsilon. 
        \end{equation*}
        Then we have
        \begin{equation*}
            \begin{aligned}
                &\norm{\nabla f(x) + \Jc(x)\lambda + \Gamma} \geq \frac{1}{\norm{\Ja(x)}} \norm{\Ja(x)\left( \nabla f(x) + \Jc(x)\lambda + \Gamma\right)}\\
                ={}&  \frac{1}{\norm{\Ja(x)}} \norm{\Ja(x)\nabla f(x) + \Gamma}\\
                ={}& \frac{1}{\norm{\Ja(x)}} \norm{\nabla h(x) + \Gamma}\geq \frac{1}{\Mxa} \mathrm{dist}\left(0,  \nabla h(x) + \NX(x) \right).
            \end{aligned}
        \end{equation*}
        This completes the proof. 
    \end{proof}

    In the following theorem, we show that for any given $x \in \K$, any first-order stationary point of \eqref{Prob_Pen} within $\Omega_x$ is feasible, hence is a first-order stationary point of \eqref{Prob_Ori}. 
    \begin{theo}
        \label{Theo_equivalence}        
        Suppose Assumption \ref{Assumption_f} and Assumption \ref{Assumption_A} hold. Then for any $x \in \K$, any $y \in \Omega_{x}$, and any $\beta \geq \beta_x$, it holds that 
        \begin{equation}
            \label{Eq_Theo_equivalence_2}
            \mathrm{dist}\left(0,  \nabla h(y) + \NX(y) \right) \geq \frac{\sigma_{x,c} \beta}{8(\Mxa + \Mxr + 1)} \norm{c(y)}.
        \end{equation}
        Therefore,  if $y\in \Omega_{x}$ is a first-order stationary point of \eqref{Prob_Pen}, then it holds that $y$ is a first-order stationary point of \eqref{Prob_Ori}. 
    \end{theo}
    \begin{proof}
        For any $y \in \Omega_{x}$, based on the constants $\Mxa$ and $\Mxr$ defined in Table \ref{Table_Constants}, and the definition of $\delta_x$ in \eqref{eq:deltax}, we have that 
        \begin{equation*}
            \norm{\left(\Ja(y) - \RA(y) - I_n \right)} \leq \Mxa + \Mxr\norm{c(y)} + 1 \leq \Mxa + \Mxr + 1. 
        \end{equation*}
        Then, together with Assumption \ref{Assumption_A} and the definitions of the constants in Table \ref{Table_Constants}, it holds that 
        \begin{equation}
            \label{Eq_Theo_equivalence_0}
            \begin{aligned}
                &\norm{P_{\E}\left(\Ja(y) - \RA(y) - I_n \right) \nabla g(y) } = \norm{P_{\E}\left(\Ja(y) - \RA(y) - I_n \right) \Ja(y)\nabla f(\A(y)) } \\
                \leq{}& \norm{P_{\E}(\Ja(y) - I_n)\Ja(y)\nabla f(\A(y))} + 
                \norm{P_{\E}\RA(y) \Ja(y) \nabla f(\A(y))}\\
                \leq{}&   \Mxf  (\Mxi + \Mxr \Mxa ) \norm{c(y)}. 
            \end{aligned}
        \end{equation}
        On the other hand, from Lemma \ref{Le_aux_relationship_c_dist_PEJcc} and Assumption \ref{Assumption_A}(2), we get 
        \begin{equation}
            \label{Eq_Theo_equivalence_1}
            \begin{aligned}
                &\norm{P_{\E}\left(\Ja(y) - \RA(y) - I_p \right) \Jc(y) c(y)}  \\
                \geq{}& \norm{P_{\E}\Jc(y) c(y)} - \norm{\RA(y)} \norm{\Jc(y)}\norm{c(y)} - \norm{\Ja(y) \Jc(y) c(y)} \\
                \geq{}& \frac{\sigmaxc}{2} \norm{c(y)} - \Mxr\Mxc  
                \norm{c(y)}^2 -\Mxb \norm{c(y)}^2
                \geq  \frac{\sigma_{x,c}}{4} \norm{c(y)}. 
            \end{aligned}
        \end{equation}
        In the last inequality, we use the fact that
        $\norm{c(y)}=\norm{c(y)-c(x)}\leq \Mxc\norm{y-x}\leq M_{x,c}\delta_x$. 
        Therefore, by combining \eqref{Eq_Theo_equivalence_0} and \eqref{Eq_Theo_equivalence_1} together, we have 
        \begin{equation*}
            \begin{aligned}
                & (\Mxa + \Mxr  + 1) \cdot \mathrm{dist}\left(0,  \nabla h(y) + \NX(y) \right) \\
                \geq{}&\mathrm{dist}\left(0, P_{\E}\left(\Ja(y) - \RA(y) - I_p \right) (\nabla h(y) + \NX(y)) \right) = \norm{P_{\E}\left(\Ja(y) - \RA(y) - I_p \right) \nabla h(y)} \\
                ={}& \norm{P_{\E}\left(\Ja(y) - \RA(y) - I_p \right) 
                (\nabla g(y) + \beta \Jc(y)c(y))  }\\
                \geq{}& \beta \norm{P_{\E}\left(\Ja(y) - \RA(y) - I_p \right) \Jc(y) c(y)} - \norm{P_{\E}\left(\Ja(y) - \RA(y) - I_p \right) \nabla g(y) }\\
                \geq{}& \frac{\sigma_{x,c} \beta}{4} \norm{c(y)} - \Mxf  (\Mxi + \Mxr\Mxa ) \norm{c(y)}\\
                \geq{}& \frac{\sigma_{x,c} \beta}{8} \norm{c(y)}. 
            \end{aligned}
        \end{equation*}
        This verifies the validity of \eqref{Eq_Theo_equivalence_2}, thus completing the proof of the first part of this theorem. 

        Furthermore, whenever $y \in \Omega_x$ is a first-order stationary point of \eqref{Prob_Pen}, we can conclude from the optimality condition of \eqref{Prob_Pen} that $0 \in \nabla h(y) + \NX(y)$.
        Then it follows from \eqref{Eq_Theo_equivalence_2} that 
        \begin{equation*}
            0 =  \mathrm{dist}\left(0,  \nabla h(y) + \NX(y) \right) \geq \frac{\sigma_{x,c} \beta}{8} \norm{c(y)} \geq 0. 
        \end{equation*}

        This implies that $\norm{c(y)} = 0$ and hence $y\in\K$. Thus, all the first-order stationary points of \eqref{Prob_Pen} are feasible for \eqref{Prob_Ori}. From Theorem \ref{Theo_equivalence_feabile}, we can conclude that these first-order stationary points of \eqref{Prob_Pen} are first-order stationary points of \eqref{Prob_Ori}. This completes the proof. 
    \end{proof}

    Moreover, we make the following assumption on the growth rate of $\norm{\A(y) - y}$. 
    \begin{assumpt}
        \label{Assumption_Ax_x_diff}
        For any $x \in \K$, there exists $\Lxres > 0$ such that the following inequality holds for any $y \in \Omegax{x}$,
        \begin{equation}
            \norm{\A(y) - y}\leq \Lxres \norm{c(y)}. 
        \end{equation}
    \end{assumpt}
    Then the following theorem shows that any $\varepsilon$-first-order stationary point of \eqref{Prob_Pen} corresponds to a $2\varepsilon$-first-order stationary point of \eqref{Prob_Ori}. 
    \begin{theo}
        \label{Theo_equivalence_approx}
        Suppose Assumption \ref{Assumption_f}, Assumption \ref{Assumption_A}, and Assumption \ref{Assumption_Ax_x_diff} hold. For any $x \in \K$, any $y \in \Omega_{x}$, and any $\beta \geq \beta_x$, suppose $y$ is an $\varepsilon$-first-order stationary point of \eqref{Prob_Pen}, then it holds that 
        \begin{equation}
            \mathrm{dist}\left(0, \nabla f(y) + \mathrm{range}(\Jc(y)) + \NX(y) \right) \leq \left(1+ \frac{8(\Mxjc \Mxf+ \Lxres)(\Mxa + \Mxr + 1)}{\sigmaxc \beta } \right)\varepsilon. 
        \end{equation}
    \end{theo}
    \begin{proof} 
        Consider $x \in \K$ and any $y \in \Omega_{x}$ that is an $\varepsilon$-first-order stationary point of \eqref{Prob_Pen}. It holds that 
        \begin{equation*}
            \begin{aligned}
                \varepsilon \geq{}& \mathrm{dist}\left(0, \Ja(y)\nabla f(\A(y)) + \beta \Jc(y)c(y) + \NX(y) \right)\\
                \geq{}& \mathrm{dist}\left(0, \Ja(y)\nabla f(\A(y)) + \mathrm{range}(\Jc(y)) + \NX(y) \right)\\
                ={}& \mathrm{dist}\left(0, P_T(y)\Ja(y)\nabla f(\A(y)) + \mathrm{range}(\Jc(y)) + \NX(y) \right)
                \\
                \geq{}& \mathrm{dist}\left(0, P_T(y) \nabla f(\A(y)) + \mathrm{range}(\Jc(y)) + \NX(y) \right) \\
                &- \norm{P_{T}(y)\Ja(y)\nabla f(\A(y)) - P_{T}(y)\nabla f(\A(y))}\\
                \geq{}& \mathrm{dist}\left(0, P_T(y) \nabla f(\A(y)) + \mathrm{range}(\Jc(y)) + \NX(y) \right) - \Mxjc \Mxf \norm{c(y)}\\
                \geq{}& \mathrm{dist}\left(0, P_T(y) \nabla f(y) + \mathrm{range}(\Jc(y)) + \NX(y) \right) - (\Mxjc \Mxf + \Lxres) \norm{c(y)}\\
                \geq{}& \mathrm{dist}\left(0,   \nabla f(y) + \mathrm{range}(\Jc(y)) + \NX(y) \right) - \frac{8(\Mxjc \Mxf + \Lxres)(\Mxa + \Mxr + 1)}{\sigmaxc \beta } \varepsilon.
            \end{aligned}
        \end{equation*}
        Here the first equality directly follows from the fact that $ (I_n - P_{T}(y))\nabla f(\A(y)) \in \mathrm{range}(\Jc(y)) + \E^{\perp}$. Moreover, the third inequality uses the triangle inequality, and the fourth inequality directly follows from the definition of $\Mxf$ and $\Mxjc$ in Table \ref{Table_Constants}. The fifth inequality uses Assumption \ref{Assumption_Ax_x_diff}. Additionally, the last inequality is guaranteed by Theorem \ref{Theo_equivalence}.

        Therefore, we can conclude that 
        \begin{equation*}
            \mathrm{dist}\left(0,   \nabla f(y) + \mathrm{range}(\Jc(y)) + \NX(y) \right) \leq \left(1+ \frac{8(\Mxjc \Mxf+ \Lxres)(\Mxa + \Mxr + 1)}{\sigmaxc \beta } \right)\varepsilon.
        \end{equation*}
        This completes the proof. 
    \end{proof}

    In the rest of this subsection, we establish the relationship between the strict complementarity for \eqref{Prob_Ori} and that for \eqref{Prob_Pen} in the following theorem. 
    \begin{theo}
        \label{Theo_equivalence_strict_complementarity}
        Suppose Assumption \ref{Assumption_f} and Assumption \ref{Assumption_A} hold. Then $x \in \K$ that satisfies strict complementarity for \eqref{Prob_Ori} if and only if it satisfies strict complementarity for \eqref{Prob_Pen}.
    \end{theo}
    \begin{proof}
        We first prove the ``if'' part of this theorem. For any $x \in \K$ that satisfies strict complementarity for \eqref{Prob_Pen}, Definition \ref{Defin_strict_complementarity} illustrates that 
        \begin{equation*}
            0 \in \nabla h(x) + \mathrm{ri}(\NX(x))= \Ja(x) \nabla f(x) + \mathrm{ri}(\NX(x)).
        \end{equation*}
        Then from Lemma \ref{Le_JC_Null}, it holds that 
        $0 \in \nabla f(x) + \mathrm{range}(\Jc(x)) + \mathrm{ri}(\NX(x))$. This completes the first part of the proof. 

        Now, we prove the ``only if'' part of this theorem. Since $x \in \K$ satisfies strict complementarity of \eqref{Prob_Ori}, it holds from Definition \ref{Defin_strict_complementarity} that $0 \in \nabla f(x) + \mathrm{range}(\Jc(x)) + \mathrm{ri}(\NX(x))$. Together with Lemma \ref{Le_JC_Null}, we can conclude that
        \begin{equation*}
            0 \in  = \Ja(x)\nabla f(x) + \mathrm{ri}(\NX(x)) = \nabla h(x) + \mathrm{ri}(\NX(x)). 
        \end{equation*}
        As a result, $x$ satisfies strict complementarity for \eqref{Prob_Pen}. This completes the proof. 
    \end{proof}

\subsubsection{Equivalence on second-order stationary points}
\label{Subsection_local_secondorder_equivalence}

    In this part, we focus on the equivalence of second-order stationary points of \eqref{Prob_Ori} and those of \eqref{Prob_Pen}. We first make the following assumption that guarantees the twice differentiability of $h$ over $\X$. 
    \begin{assumpt}
        \label{Assumption_twice_differentiable}
        \begin{enumerate}
            \item The objective function $f$ is twice-differentiable over $\Y$. 
            \item The constraint mapping $c$ and the constraint dissolving mapping $\A$ are twice-differentiable over $\Y$. 
        \end{enumerate}
    \end{assumpt}

    The following lemma states that \eqref{Prob_Ori} and \eqref{Prob_Pen} have the same multiplier $\Gamma$ corresponding to the constraint $x \in \X$. Its proof directly follows from Lemma \ref{Le_JC_Null}, and is omitted for simplicity. 
    \begin{lem}\label{Le_unique_Gamma}
         Suppose Assumption \ref{Assumption_f} and Assumption \ref{Assumption_A} hold. Suppose $x \in \K$  is a first-order stationary point of \eqref{Prob_Pen}. Then there exists $\Gamma \in \NX(x)$ such that 
         \begin{equation}
             \nabla h(x) + \Gamma = 0
         \end{equation}
         if and only if for $\Gamma\in \NX(x)$, there exists $\lambda\in \Rp$ that satisfies 
         \begin{equation}
             \nabla f(x) + \Jc(x) \lambda + \Gamma = 0. 
         \end{equation}
    \end{lem}

    Next, we present the following three auxiliary lemmas.

    \begin{lem}\label{Le_SONC_aux1}
        Suppose Assumption \ref{Assumption_f}, Assumption \ref{Assumption_A}, and Assumption \ref{Assumption_twice_differentiable} hold. Then for any given $x \in \K$, any $i \in [p]$, and any $d \in \mathrm{null}(\Jc(x)\tp) \cap \E$, it holds that 
        \begin{equation}\label{eq:MH}
            0 =  \Ja(x) \nabla^2 c_i(x) \Ja(x)\tp d+ \nabla^2 \A(x)[\nabla c_i(x)] d.
        \end{equation}
    \end{lem}
    \begin{proof}
        For any $i\in[p]$ and any $u \in \Rn$, we consider the auxiliary function $\phi(y) = \inner{u,  \Ja(y) \nabla c_i(\A(y))}$ with $y \in \X$. 
        Then $\inner{\nabla \phi(y),v} = \inner{u,
        \Ja(y) \nabla^2 c_i(\A(y)) \Ja(y)\tp v  +\nabla^2\A(y)[\nabla c_i(\A(y))] v}$ for any $v\in\bb{R}^n$.
        Observe from Assumption \ref{Assumption_A} that $\phi(y) = 0$ holds for any $y \in \K$.
        Thus for any $x \in \K$ and any $d \in \mathrm{null}(\Jc(x)\tp)\cap \E = (\mathrm{range}(\Jc(x)) + \E^{\perp} )^\perp$, it follows from Lemma \ref{Le_SONC_aux0} that 
        \begin{equation*}
            \begin{aligned}
                &0= \inner{d, \nabla \phi(x)} = \inner{u,  \Ja(x) \nabla^2 c_i(x) \Ja(x)\tp d+  \nabla^2 \A(x)[\nabla c_i(x)] d}.              
            \end{aligned}
        \end{equation*}
        From the arbitrariness of $u \in \Rn$, we get that $0 =  \Ja(x) \nabla^2 c_i(x) \Ja(x)\tp d+ \nabla^2 \A(x)[\nabla c_i(x)] d$. This completes the proof.
    \end{proof}

    \begin{lem}
        \label{Le_SONC_aux3}
        Suppose Assumption \ref{Assumption_f}, Assumption \ref{Assumption_A}, and Assumption \ref{Assumption_twice_differentiable} hold. Then for any given $x \in \K$, any $d \in \mathrm{null}(\Jc(x)\tp) \cap \E$, any $w \in \Rn$, it holds that 
        \begin{equation*}
            \nabla^2 \A(x)[w - P_{\E}w] d = 0.
        \end{equation*}
    \end{lem}
    \begin{proof}
        For any $u \in \Rn$ and any $v \in \E^{\perp}$, let the auxiliary function $\phi(y)$ be defined as $\phi(y) = \inner{u, \Ja(y) v - v } $. Notice that $v \in \E^{\perp} \subseteq \NX(x)$ by Lemma \ref{Le_aux_define_E}, then it holds that $\phi(y) = 0$ for any $y \in \K$. Thus for any given $x \in \K$ and any $d \in \mathrm{null}(\Jc(x)\tp) \cap \E$, it follows from Lemma \ref{Le_SONC_aux0} that 
        \begin{equation*}
            0=\inner{d, \nabla^2 \A(x)[v]u} = \inner{u, \nabla^2 \A(x)[v]d}.
        \end{equation*}
        From the arbitrariness of $u\in \Rn$, it holds that $\nabla^2 \A(x)[v]d = 0$. Notice that for any $w \in \Rn$, it holds that $w - P_{\E}w \in \E^{\perp}$. This completes the proof. 
    \end{proof}

    \begin{lem}\label{Le_SONC_aux2}
        Suppose Assumption \ref{Assumption_f}, Assumption \ref{Assumption_A}, and Assumption \ref{Assumption_twice_differentiable} hold. Then for any given $x \in \K$, any $d \in \mathrm{null}(\Jc(x)\tp) \cap \E$, any $v \in \Rn$, it holds that 
        \begin{equation*}
            P_{\E}\nabla^2 \A(x)[\Ja(x)v] d =0.
        \end{equation*}
    \end{lem}
    \begin{proof}
        For any $u, w \in \Rn$, let the auxiliary function $\phi$ be defined as $\phi(y)= \inner{u, P_{\E}(\Ja(y) - I_n) \Ja(y)w}$ for $y\in\X$. Then we have that  $\inner{\nabla \phi(y),v} = 
        \inner{u, 2P_{\E}\nabla^2\A(y)[\Ja(y)w] v-P_{\E}\nabla^2\A(y)[w] v}$ for any $v\in\bb{R}^n$.
        From Lemma \ref{Le_Ja_ideo}, we have that $\phi(y) = 0$ holds for any $y \in \K$. Thus for any given $x \in \K$ and any $d \in \mathrm{null}(\Jc(x)\tp) \cap \E$, it follows from Lemma \ref{Le_SONC_aux0} that 
        \begin{equation*}
            0 = \inner{d, \nabla \phi(x)} = 
            \inner{u, 2P_{\E}\nabla^2\A(x)[\Ja(x)w]d  - P_{\E}\nabla^2\A(x)[w] d   }.
        \end{equation*}
        From the arbitrariness of $u\in \Rn$, it holds that $P_{\E}\nabla^2 \A(x)[2\Ja(x)w -w] d = 0$ for any $w \in \Rn$. Together with Lemma \ref{Le_SONC_aux3}, it holds that $P_{\E}\nabla^2 \A(x)[2P_{\E}\Ja(x)w -P_{\E}w] d = 0$. Then for any $v \in \Rn$, let $w = \Ja(x)v$, it follows from Lemma \ref{Le_SONC_aux3} and Lemma \ref{Le_Ja_ideo} that 
        \begin{equation*}
            \begin{aligned}
                &P_{\E}\nabla^2 \A(x)[\Ja(x)v] d = P_{\E}\nabla^2 \A(x)[P_{\E}\Ja(x)v] d \\
                ={}& P_\E \nabla^2 \A(x)[2P_{\E}\Ja(x) (\Ja(x) v) - P_{\E}(\Ja(x)v)] d = 0.
            \end{aligned}
        \end{equation*}
        This completes the proof. 
    \end{proof}

    In the following theorem, we prove that any second-order stationary point of \eqref{Prob_Pen} on $\K$ is also a second-order stationary point of \eqref{Prob_Ori}. 
    \begin{theo}
        Suppose Assumption \ref{Assumption_f}, Assumption \ref{Assumption_A}, and Assumption \ref{Assumption_twice_differentiable} hold. Then for any $x \in \K$ that is a second-order stationary point of \eqref{Prob_Pen},  $x$ is a second-order stationary point of \eqref{Prob_Ori}. 
    \end{theo}
    \begin{proof}
        For any $x \in \K$ that is a second-order stationary point of \eqref{Prob_Pen}, Definition \ref{Defin_SONC_Pen} illustrates that there exists $\Gamma \in \NX(x)$ such that  
        \begin{equation*}
            \inner{d, \nabla^2 h(x)d} - \zeta(\Gamma, \ca{T}^2_{\X}(x,d)) \geq 0, \quad \forall d \in {\cal C}_1(x). 
        \end{equation*}
        Moreover, Lemma \ref{Le_unique_Gamma} illustrates that there exists $\lambda \in \Rp$ such that $(\lambda, \Gamma) \in \M_2(x)$. For any $d \in \ca{C}_2(x)$, from the definition of $\ca{C}_2(x)$ and Lemma \ref{Le_aux_linNK}, it holds that $d \in \mathrm{range}(\ca{T}_{\K}(x)) = \mathrm{null}(\Jc(x)\tp) \cap \E$. Therefore, it holds that 
        \begin{equation*}
            \begin{aligned}
                0 \leq{}& \inner{d, \nabla^2 h(x)d} - \zeta(\Gamma, \ca{T}^2_{\X}(x,d))\\
                ={}& \inner{d, \left(\Ja(x)\nabla^2 f(x)\Ja(x)\tp + \nabla^2 \A(x)[\nabla f(x)]\right)d} - \zeta(\Gamma, \ca{T}^2_{\X}(x,d))\\
                ={}& \inner{d, \left(\Ja(x)\nabla^2 f(x)\Ja(x)\tp - \nabla^2 \A(x)[\Jc(x)\lambda + \Gamma]\right)d} - \zeta(\Gamma, \ca{T}^2_{\X}(x,d))\\
                ={}& \inner{d, \left(\Ja(x)\nabla^2 f(x)\Ja(x)\tp - \nabla^2 \A(x)[\Jc(x)\lambda]\right)d} - \zeta(\Gamma, \ca{T}^2_{\X}(x,d))\\
                ={}& \inner{d, \left(\Ja(x)\nabla^2 f(x)\Ja(x)\tp + \Ja(x)\nabla^2 c(x)[\lambda]\Ja(x)\tp \right)d} - \zeta(\Gamma, \ca{T}^2_{\X}(x,d))\\
                ={}& \inner{d, \left(\nabla^2 f(x) + \nabla^2 c(x)[\lambda]\right)d} - \zeta(\Gamma, \ca{T}^2_{\X}(x,d)).
            \end{aligned}
        \end{equation*}
        where the second equality follows from Lemma \ref{Le_unique_Gamma}, the third equality comes from Lemma \ref{Le_SONC_aux2} and the fact that $\Ja(x)\Gamma = \Gamma$ by 
        Assumption \ref{Assumption_A}(3b),
        the forth equality is by using Lemma \ref{Le_SONC_aux1} and the last one comes from Lemma \ref{Le_range_JAtp}. 
        Therefore, we can conclude that $x$ is a second-order stationary point of \eqref{Prob_Ori}. 
    \end{proof}

    In the rest of this subsection, we aim to show that all second-order stationary points of \eqref{Prob_Ori} are second-order stationary points of \eqref{Prob_Pen} for a sufficiently large but finite penalty parameter $\beta$. We begin our theoretical analysis with the following assumption on  \eqref{Prob_Ori}. 
    \begin{assumpt}
        \label{Assumption_RobinsonCQ}
        The set $\X$ is second-order directionally differentiable. 
    \end{assumpt}
    As illustrated in \cite[Definition 3.1]{mohammadi2021parabolic}, all parabolically regular sets are second-order directionally differentiable. As a result, the Assumption \ref{Assumption_RobinsonCQ} can be satisfied for a wide range of closed convex sets, including the unit ball in $\ell_1$- or $\ell_2$-norm, a box, the semi-positive definite (SPD) cone, and the second-order cone.

    We present the following auxiliary lemma on the constraint qualification of the constraint $\A(x) \in \X$. 
    \begin{lem}
        \label{Le_AxinX_RCQ}
        Suppose Assumption \ref{Assumption_f} and Assumption \ref{Assumption_A} hold. Then for any $x \in \K$, the constraint $\A(x) \in \X$ satisfies the constraint nondegeneracy condition. 
    \end{lem}
        \begin{proof}
        From Assumption \ref{Assumption_f}(2), we have $\mathrm{dim}(\Jc(x)\tp \E) = r$, then it holds that $\mathrm{dim}(\mathrm{null}(\Jc(x)\tp ) \cap \E) = \mathrm{dim}(\E) - r$. Moreover, it follows from Assumption \ref{Assumption_A}(3b) and Lemma \ref{Le_range_JAtp} that $\inner{d,\Ja(x)\tp d} = \norm{d}^2$ holds for any $d \in \left(\mathrm{null}(\Jc(x)\tp ) \cap \E\right) + \E^{\perp}$. As a result, we can conclude that $\mathrm{dim}(\Ja(x)\Rn) \geq \mathrm{dim}(\mathrm{null}(\Jc(x)\tp ) \cap \E) + \mathrm{dim}(\E^{\perp}) =  n-r$.
        
        Moreover, notice that $\mathrm{range}(\Jc(x)) \subseteq \mathrm{null}(\Ja(x))$ by Assumption \ref{Assumption_A}(2) , and $\mathrm{dim}(\mathrm{range}(\Jc(x))) \geq r$ by Assumption \ref{Assumption_f}(2a), then we can conclude that 
        \begin{equation*}
            r \leq \mathrm{dim}\Big(\mathrm{range}(\Jc(x))\Big)  \leq \mathrm{dim}\Big(\mathrm{null}(\Ja(x))\Big) = n- \mathrm{dim}\Big( \mathrm{range}(\Ja(x))\Big) \leq  r.
        \end{equation*}
        Therefore, we can conclude that $\mathrm{null}(\Ja(x)) = \mathrm{range}(\Jc(x))$ and thus
        \begin{equation}
            \label{Eq_Le_AxinX_RCQ_0}
            \Ja(x)\tp \Rn = \mathrm{range}(\Ja(x)\tp) = \Big(\mathrm{null}(\Ja(x))\Big)^{\perp} = \mathrm{null}(\Jc(x)\tp).
        \end{equation}

        Furthermore, it follows from Lemma \ref{Le_aux_unique_Gamma} that $\mathrm{range}(\Jc(x)) \cap \mathrm{range}(\NX(x)) = \{0\}$. Then together with \eqref{Eq_Le_AxinX_RCQ_0}, we can conclude that 
        \begin{equation*}
            \Rn = \Big( \mathrm{range}(\Jc(x)) \cap \mathrm{range}(\NX(x)) \Big)^{\perp} = \mathrm{null}(\Jc(x)\tp) + \mathrm{lin}(\TX(x)) = \Ja(x)\tp \Rn + \mathrm{lin}(\TX(x)).
        \end{equation*}
        This completes the proof. 
    \end{proof}
Based on Definition \ref{Defin_SO_directionally_diff}, we present the following auxiliary lemma based on the chain rule discussed in \cite[Section 2.4.4]{bonnans2013perturbation}. 
    \begin{lem}
        \label{Le_TX2_innerprod}
        Suppose Assumption \ref{Assumption_f}, Assumption \ref{Assumption_A}, Assumption \ref{Assumption_twice_differentiable}, and Assumption \ref{Assumption_RobinsonCQ} hold. Then for any $x \in \K$, 
        and any $w \in \mathrm{range}(\ca{T}_{\X}(x))$, let $w = w_1 + w_2$ for $w_1 \in \mathrm{null}(\Jc(x)\tp) \cap \E$ and $w_2 \in \mathrm{range}(\Jc(x)) \cap \E$, it holds that 
        \begin{equation*}
            \inner{\Gamma, \ca{T}_{\X}^2(x, \Ja(x)\tp w)} = \inner{\Gamma, \ca{T}_{\X}^2(x, w)} + \inner{w_2, \nabla^2 \A(x)[\Gamma] w_2}, \quad \forall\; 
            \Gamma \in \NX(x).
        \end{equation*}
    \end{lem}

    \begin{proof}
        Consider the constraint $\A(x)\in \X$. From Lemma \ref{Le_AxinX_RCQ} and \cite[Proposition 3.33]{bonnans2013perturbation}, it holds that 
        \begin{equation*}
            \Ja(x)\tp \ca{T}_{\X}^2(x, w) = \ca{T}_{\X}^2(x, \Ja(x)\tp w) - \nabla^2 \A(x)^*[w, w].
        \end{equation*}
        Here $\nabla^2 \A(x)^*$ refers to the adjoint mapping of $\nabla^2 \A(x)$, in the sense that $\inner{\nabla^2 \A(x)^*[z_1, z_2], z_3} = \inner{z_2, \nabla^2 \A(x)[z_3] z_1}$ for any $z_1, z_2 \in \Rn$ and $z_3 \in \Rp$. It follows from Lemma \ref{Le_unique_Gamma} and Lemma \ref{Le_Subdifferential_h} that $\Gamma = -\Ja(x)\nabla f(x)$. Together with Lemma \ref{Le_SONC_aux2} and the fact that $\Ja(x) \Gamma= \Gamma$ by Assumption \ref{Assumption_A}(3b), it holds that 
        \begin{equation*}
            \inner{\Gamma, \nabla^2 \A(x)^*[w, w]} = \inner{w_1+w_2, \nabla^2 \A(x)[\Gamma] (w_1+w_2)} =  \inner{w_2, \nabla^2 \A(x)[\Gamma] w_2}.
        \end{equation*}
        As a result, we can conclude that 
        \begin{equation*}
            \begin{aligned}
                &\inner{\Gamma, \ca{T}_{\X}^2(x, \Ja(x)\tp w)} = \inner{\Gamma, \Ja(x)\tp \ca{T}_{\X}^2(x, w) + \nabla^2 \A(x)^*[w, w]}\\
                ={}& \inner{\Gamma, \ca{T}_{\X}^2(x, w)} + \inner{w_2, \nabla^2 \A(x)[\Gamma] w_2}. 
            \end{aligned}
        \end{equation*}
        This completes the proof. 
    \end{proof}

    Next we present the following theorem illustrating that any second-order stationary point of \eqref{Prob_Ori} is a second-order stationary point of \eqref{Prob_Pen}, for a sufficiently large but finite penalty parameter $\beta$. 
    \begin{theo}
        \label{Theo_equivalence_SONC}
        Suppose Assumption \ref{Assumption_f}, Assumption \ref{Assumption_A}, Assumption \ref{Assumption_twice_differentiable}, and Assumption \ref{Assumption_RobinsonCQ} hold. Then for any $x \in \K$ that is a second-order stationary point of \eqref{Prob_Ori}, it holds that $x$ is a second-order stationary point of \eqref{Prob_Pen} with $\beta \geq \beta_x$. 
    \end{theo}
    \begin{proof}
        For any $x \in \K$ that is a second-order stationary point of \eqref{Prob_Ori},  there exists $(\lambda,\Gamma)\in{\cal M}_2(x)$ such that $0=\nabla f(x) + \Jc(x)\lambda + \Gamma$ and
        \begin{equation*}
            \inner{d, \left(\nabla^2 f(x) + \nabla^2 c(x)[\lambda]\right)d} - \zeta(\Gamma, \ca{T}^2_{\X}(x,d)) \geq 0 \quad 
            \forall \; d \in {\cal C}_2(x). 
        \end{equation*}
        Recall that $\nabla^2 h(x) = \Ja(x)\nabla^2f(x)\Ja(x)^\top + \nabla^2\A(x)[\nabla f(x)] + \beta \Jc(x)\Jc(x)^\top$ by Lemma \ref{Le_Subdifferential_h} and $\Gamma = -\Ja(x)\nabla f(x)$ by Lemma \ref{Le_unique_Gamma}. Consider any $d \in \E$ that is orthogonally decomposed as $d = d_1 + d_2$ with $d_1\in \mathrm{range}(\ca{T}_{\K}(x))$ and 
        $d_2 \in \mathrm{range}(\ca{T}_{\K}(x))^{\perp} \cap \E$. It is worth mentioning that the inclusion $\K \subseteq \X$ guarantees the existence of such an orthogonal decomposition. Additionally, from Lemma \ref{Le_aux_linNK}, it holds that $d_1 \in \mathrm{null}(\Jc(x)\tp) \cap \E $ and $d_2 \in \mathrm{range}(\Jc(x)) \cap \E$.

        Then we have that $\Ja(x)^\top d_1 = d_1$ by 
        Lemma \ref{Le_range_JAtp}, 
        $\nabla^2 \A(x)[\Gamma] d_1 = \nabla^2 \A(x)[-\Ja(x)\nabla f(x)] d_1 \in \E^{\perp}$ by Lemma \ref{Le_SONC_aux2}, and 
        $$
         \nabla^2 \A(x) [\Jc(x)\lambda] d_1=-\Ja(x)\nabla^2 c(x)[\lambda] \Ja(x)^\top d_1 
        $$
        by
        Lemma \ref{Le_SONC_aux1}.
        Thus 
        \begin{equation*}
           \begin{aligned}
               &\nabla^2 h(x) d_1 = \Ja(x)\nabla^2 f(x) \Ja(x)\tp d_1 
            + \nabla^2 \A(x) [\nabla f(x)] d_1 \\
            ={}& \Ja(x) \nabla^2 f(x) \Ja(x)\tp d_1 + \nabla^2 \A(x)[-\Jc(x)\lambda - \Gamma] \Ja(x)\tp d_1 \\
            \subseteq{}& \Ja(x)\big(\nabla^2 f(x) + \nabla^2 c(x)[\lambda]\big)  \Ja(x)\tp d_1 + \E^{\perp}.
           \end{aligned}
        \end{equation*}
        As a result, it holds that 
        \begin{equation}\label{Eq_Theo_equivalence_SONC_5}
            \begin{aligned}
                \inner{d_i, \nabla^2 h(x) d_1}  = \inner{d_i, \Ja(x)\left(\nabla^2 f(x) + \nabla^2 c(x)[\lambda]\right)\Ja(x)\tp d_1}, \quad i = 1,2. 
            \end{aligned}
        \end{equation}

        Notice that for any $w \in \E^{\perp}\subseteq \NX(x) $, it follows from Assumption \ref{Assumption_A}(3) that $\inner{w, \Ja(x)^\top d_2} = \inner{\Ja(x)w, d_2} = \inner{w, d_2} = 0$. Then it holds that $\Ja(x)^\top d_2\in \E$ from the arbitrariness of $w \in \E^{\perp}$. Together with Lemma \ref{Le_range_JAtp}, we get $\Ja(x)^\top d_2\in \mathrm{null}(\Jc(x)^\top) \cap \E$. As a result, we can conclude that 
        \begin{equation}
            \label{Eq_Theo_equivalence_SONC_0}
            \begin{aligned}
                &\inner{d_2, \nabla \A(x) \nabla^2 \A(x)[\nabla f(x)] \nabla \A(x)\tp d_2}
                =-\inner{\nabla \A(x)\tp d_2, \nabla^2 \A(x)[\Jc(x) \lambda + \Gamma] \nabla \A(x)\tp d_2}\\
                ={}& - \inner{\nabla \A(x)\tp d_2, \nabla^2 \A(x)[\Jc(x) \lambda] \nabla \A(x)\tp d_2}\\
                ={}& \inner{(\nabla \A(x)\tp)^2  d_2, \DJc(x)[\lambda] (\nabla \A(x)\tp)^2  d_2}=\inner{(\nabla \A(x)\tp)^2  P_{\E}d_2, \DJc(x)[\lambda] (\nabla \A(x)\tp)^2 P_{\E}d_2)}\\
                ={}& \inner{ \nabla \A(x)\tp  P_{\E}d_2,  \DJc(x)[\lambda] \nabla \A(x)\tp P_{\E} d_2}= \inner{d_2,  \nabla \A(x)\DJc(x)[\lambda] \nabla \A(x)\tp d_2}. 
            \end{aligned}
        \end{equation}
        Here the first equality follows from the first-order optimality condition of \eqref{Prob_Ori}, the second equality follows from Lemma \ref{Le_SONC_aux2}, and the third equality uses Lemma \ref{Le_SONC_aux1}. The fourth equality uses $P_{\E}d_2 = d_2$ and the fifth equality directly follows from Lemma \ref{Le_Ja_ideo}. The last equality uses again the fact that $P_{\E} d_2 = d_2$. 
        
        Therefore, we have the following estimation for the lower-bound of $\inner{d_2, \nabla^2 h(x) d_2}$, 
        \begin{equation}\label{Eq_Theo_equivalence_SONC_4}
            \begin{aligned}
                &\inner{d_2, \nabla^2 h(x) d_2} \\
                ={}& \beta \inner{d_2, \Jc(x) \Jc(x)\tp d_2} + \inner{d_2, \Ja(x) \nabla^2 f(x) \Ja(x)^\top d_2}
                +\inner{d_2,\nabla^2\A(x)[\nabla f(x)] d_2}
                \\
                \geq{}& \beta \sigmaxc^2 \norm{d_2}^2+\inner{d_2, \Ja(x) \nabla^2 f(x) \Ja(x)^\top d_2} - \Lxa \Mxf \norm{d_2}^2 \\
                \geq{}& \inner{d_2, \Ja(x) \nabla^2 f(x) \Ja(x)^\top d_2}  + \Lxa \Mxa^2 \Mxf\norm{d_2}^2  + \Lxa\Mxa \Mxf \norm{d_2}^2 \\
                \geq{}& \inner{d_2, \Ja(x)\left(\nabla^2 f(x) + \nabla^2 \A(x)[\nabla f(x)]\right)\Ja(x)\tp d_2}  + \left| \inner{d_2, \nabla^2 \A(x)[\Gamma] d_2} \right|\\
                \geq{}& \inner{d_2, \Ja(x)\left(\nabla^2 f(x) +\nabla^2 c(x)[\lambda]\right)\Ja(x)\tp d_2} + \left|\inner{d_2, \nabla^2 \A(x)[\Gamma] d_2}\right|.
            \end{aligned}
        \end{equation}
        Here the first inequality follows from the definition of $\sigmaxc$ in Table \ref{Table_Constants}, the second inequality follows from the choice of $\beta_x$ \eqref{Eq_betax_defin} with the fact that $\Ja(x) \nabla f(x) + \Gamma = 0$, and the fourth inequality directly follows from \eqref{Eq_Theo_equivalence_SONC_0}.

        As a result, we can conclude that 
        \begin{equation}
            \label{Eq_Theo_equivalence_SONC_6}
            \begin{aligned}
                &\inner{d, \nabla^2 h(x) d} = \inner{(d_1 + d_2), \nabla^2 h(x) (d_1 + d_2)}\\
                ={}& \inner{d_1, \nabla^2 h(x) d_1} + \inner{d_2, \nabla^2 h(x) d_2} + 2 \inner{d_2, \nabla^2 h(x) d_1}\\
                \geq{}& \inner{d_1 + d_2, \Ja(x)\left(\nabla^2 f(x) +\nabla^2 c(x)[\lambda]\right)\Ja(x)\tp (d_1 + d_2)}  + \left|\inner{d_2, \nabla^2 \A(x)[\Gamma] d_2}\right|\\
                ={}& \inner{\Ja(x)\tp d,\left(\nabla^2 f(x) + \nabla^2 c(x)[\lambda]\right)\big(\Ja(x)\tp d\big)}  + \left|\inner{d_2, \nabla^2 \A(x)[\Gamma] d_2}\right|,
            \end{aligned}
        \end{equation}
        where the first inequality comes from \eqref{Eq_Theo_equivalence_SONC_5} and \eqref{Eq_Theo_equivalence_SONC_4}. 
        Furthermore, from Lemma \ref{Le_TX2_innerprod}, it holds that 
        \begin{equation*}
            \begin{aligned}
                &\zeta(\Gamma, \ca{T}^2_{\X}(x, d )) = 
                \sup \{ \inner{\Gamma, u} : u\in \ca{T}^2_{\X}(x, d )\}
                \\
                ={}& -\inner{d_2, \nabla^2 \A(x)[\Gamma]d_2} + 
                \sup \{ \inner{\Gamma, v} : v\in \ca{T}^2_{\X}(x, \Ja(x)\tp d )\} \\
                ={}& -\inner{d_2, \nabla^2 \A(x)[\Gamma]d_2} + \zeta(\Gamma, \ca{T}^2_{\X}(x, \Ja(x)\tp d ))\\
                \leq{}& \left|\inner{d_2, \nabla^2 \A(x)[\Gamma]d_2}\right| + \zeta(\Gamma, \ca{T}^2_{\X}(x, \Ja(x)\tp d )). 
            \end{aligned}
        \end{equation*}     
        Therefore, from \eqref{Eq_Theo_equivalence_SONC_6} and the fact that $\Ja(x) \nabla f(x) + \Gamma = 0$, it holds for any $d \in \ca{C}_1(x)\subseteq{\cal T}_{\cal X}(x)$ that 
        \begin{equation*}
            \begin{aligned}
                &\inner{d, \nabla^2 h(x)d} - \zeta(\Gamma, \ca{T}^2_{\X}(x,d))\\
                \geq{}& \inner{\Ja(x)\tp d,\left(\nabla^2 f(x) + \nabla^2 c(x)[\lambda]\right)\big(\Ja(x)\tp d\big)}   + \left|\inner{d_2, \nabla^2 \A(x)[\Gamma] d_2}\right| - \zeta(\Gamma, \ca{T}^2_{\X}(x,d))\\
                \geq{}& \inner{\Ja(x)\tp d,\left(\nabla^2 f(x) + \nabla^2 c(x)[\lambda]\right)\big(\Ja(x)\tp d\big)}   - \zeta(\Gamma, \ca{T}^2_{\X}(x, \Ja(x)\tp d))\\
                \geq{}& 0.
            \end{aligned}
        \end{equation*}
        This completes the proof. 
    \end{proof}

    Finally, we have the following theorems characterizing the SOSC and strong SOSC between \eqref{Prob_Ori} and \eqref{Prob_Pen}. The proofs of the following theorems are exactly the same as Theorem \ref{Theo_equivalence_SONC}, hence is omitted for simplicity. 
    \begin{theo}
        \label{Theo_equivalence_SOSC}
        Suppose Assumption \ref{Assumption_f}, Assumption \ref{Assumption_A}, and Assumption \ref{Assumption_twice_differentiable} hold. Then for any $x \in \K$ that satisfies SOSC of \eqref{Prob_Pen}, we can conclude that $x$ satisfies SOSC of \eqref{Prob_Ori}.

        Suppose we further assume that $\beta \geq \beta_x$ and the validity of Assumption \ref{Assumption_RobinsonCQ}. Then for any $x \in \K$ that satisfies SOSC of \eqref{Prob_Ori}, we can conclude that $x$ satisfies SOSC of \eqref{Prob_Pen}.
    \end{theo}

    \begin{theo}
        \label{Theo_equivalence_strong_SOSC}
        Suppose Assumption \ref{Assumption_f}, Assumption \ref{Assumption_A}, and Assumption \ref{Assumption_twice_differentiable} hold. Then for any $x \in \K$ that satisfies strong SOSC of \eqref{Prob_Pen}, we can conclude that $x$ satisfies strong SOSC of \eqref{Prob_Ori}.

        Suppose we further assume that $\beta \geq \beta_x$ and the validity of Assumption \ref{Assumption_RobinsonCQ}. Then for any $x \in \K$ that satisfies strong SOSC of \eqref{Prob_Ori}, we can conclude that $x$ satisfies strong SOSC of \eqref{Prob_Pen}.
    \end{theo}

    \subsection{Error bound condition for global equivalence}
    \label{Subsection_33}
    \label{Subsection_global_equivalence}
    In this subsection, we aim to show that under appropriate conditions, the equivalence between \eqref{Prob_Ori} and \eqref{Prob_Pen} holds globally over $\X$. We first introduce the following assumptions on the constraints of \eqref{Prob_Ori}. 
    \begin{assumpt}
        \label{Assumption_error_bound}
        \begin{enumerate}
            \item $\X$ is a compact subset of $\Rn$. 
            \item There exists a constant $\nu > 0$ such that for any $x \in \X$, it holds that 
        \begin{equation}\label{eq:errorbd}
            \mathrm{dist}\left(\Jc(x)c(x), -\NX(x)  \right) \geq \nu\norm{c(x)}. 
        \end{equation}
        \end{enumerate}
    \end{assumpt}

    \begin{rmk}
        Finding a feasible solution for optimization problems with nonconvex constraints can be challenging in general. To address this issue, it is common to impose some constraint qualification conditions or nondegeneracy conditions on the constraints. Among these conditions, Assumption \ref{Assumption_error_bound}(2) is a widely employed nondegeneracy condition for infeasible methods, which is important to ensure their sequential convergence to the feasible region. These infeasible methods include the augmented Lagrangian methods \cite{sahin2019inexact,li2021rate,li2023stochastic,xiao2024developing} and the sequential quadratic programming methods \cite{berahas2021sequential,curtis2024worst}. Moreover, it is easy to verify that the strong LICQ condition \cite[Assumption 2.1]{berahas2021sequential} ensures the validity of Assumption \ref{Assumption_error_bound}(2). Therefore, we can conclude that Assumption \ref{Assumption_error_bound}(2) is acceptable. 
    \end{rmk}

    With Assumption \ref{Assumption_error_bound}, we denote some constants as follows for establishing the global equivalence between \eqref{Prob_Ori} and \eqref{Prob_Pen} over $\X$, 
    \begin{equation}
        \label{Eq_constants_global}
        \begin{aligned}
            &\tilde{\delta} := \inf_{x \in \K} \delta_x, \quad \tilde{M}_f := \sup_{x \in \X} \Mxf, \quad \tilde{M}_A := \sup_{x \in \X} \Mxa,\\
            &\tilde{\sigma}_c = \inf_{x \in \K} \sigmaxc, \quad \tilde{\beta} := \max\left\{ \frac{\tilde{M}_f\tilde{M}_A}{\tilde{\delta}\tilde{\sigma}_c \nu}, \sup_{x \in \K} \beta_x\right\}. 
        \end{aligned}
    \end{equation}
    
    It is worth mentioning that all the constants in \eqref{Eq_constants_global} are independent of the choices of $x \in \K$. Then, based on these constants, the following theorem illustrates that under Assumption \ref{Assumption_error_bound}, any first-order stationary point of \eqref{Prob_Pen} is a first-order stationary point of \eqref{Prob_Ori}. 
    \begin{theo}
        \label{Theo_equivalence_global_FOSP}
        Suppose Assumption \ref{Assumption_f}, Assumption \ref{Assumption_A}, and Assumption \ref{Assumption_error_bound} hold. Then for any $\beta> \tilde{\beta}$ and any $x \in \X$, it holds that $x$ is a first-order stationary point of \eqref{Prob_Pen} if and only if $x$ is a first-order stationary point of \eqref{Prob_Ori}.
    \end{theo}
    \begin{proof}
        We first prove the ``only if'' part of this theorem. For any $w \in \X$ that is a first-order stationary point of \eqref{Prob_Pen}, it holds that 
        \begin{equation*}
            \begin{aligned}
                0 \in \nabla h(w) = \Ja(w) \nabla f(\A(w)) + \beta \Jc(w) c(w) + \NX(w).
            \end{aligned}
        \end{equation*}
        Therefore, 
        \begin{equation*}
            \begin{aligned}
                &0 = \mathrm{dist}\left(0, \Ja(w) \nabla f(\A(w)) +  \beta \Jc(w) c(w) + \NX(w) \right)\\
                \geq{}& \mathrm{dist}\left(0,  \beta \Jc(w) c(w) + \NX(w) \right) - \norm{\Ja(w)\nabla f(\A(w))} \\
                \geq{}& \nu \beta \norm{c(w)} -\tilde{M}_f\tilde{M}_A,
            \end{aligned}
        \end{equation*}
        where the second inequality follows from \eqref{eq:errorbd}. 
        Therefore, with $\beta \geq \tilde{\beta}$, we can conclude that 
        \begin{equation*}
            \mathrm{dist}(w, \K) \leq \frac{2}{\tilde{\sigma}_c}\norm{c(w)} \leq \frac{\tilde{M}_f \tilde{M}_A}{\tilde{\sigma}_c\nu \beta} \leq  \tilde{\delta}. 
        \end{equation*}
        As a result, there exists $x \in \K$ such that $w \in \Omegax{x}$. 
        Together with Theorem \ref{Theo_equivalence}, we can conclude that $w \in \K$ and is a first-order stationary point of \eqref{Prob_Ori}. 
        
        Following the same technique and Theorem \ref{Theo_equivalence_feabile}, we can show the validity of the ``if'' part of this theorem. This completes the proof. 
    \end{proof}

Furthermore, we have the following corollary illustrating the equivalence on second-order stationary points between \eqref{Prob_Ori} and \eqref{Prob_Pen}. The proof of Corollary \ref{Coro_equivalence_global_SO} directly follows from the proof of Theorem \ref{Theo_equivalence_global_FOSP} and Theorem \ref{Theo_equivalence_SONC}-\ref{Theo_equivalence_strong_SOSC}, hence is omitted for simplicity. 
\begin{coro}
    \label{Coro_equivalence_global_SO}
    Suppose Assumption \ref{Assumption_f}, Assumption \ref{Assumption_A}, Assumption \ref{Assumption_twice_differentiable}, and Assumption \ref{Assumption_error_bound} hold, and $\X$ is second-order directionally differentiable. Then for any $\beta \geq \tilde{\beta}$ and any $x \in \X$, it holds that
    \begin{enumerate}
        \item $x$ is a second-order stationary point of \eqref{Prob_Pen} if and only if $x$ is a second-order stationary point of \eqref{Prob_Ori};
        \item $x$ is a SOSC point of \eqref{Prob_Pen} if and only if $x$ is a SOSC point of \eqref{Prob_Ori};
        \item $x$ is a strong SOSC point of \eqref{Prob_Pen} if and only if $x$ is a strong SOSC point of \eqref{Prob_Ori}. 
    \end{enumerate}
\end{coro}

    \section{Implementations on solving \eqref{Prob_Ori} through \eqref{Prob_Pen}}
    
    In this section, we discuss how to solve \eqref{Prob_Ori} through \eqref{Prob_Pen}. In Section \ref{Subsection_implementation_framework}, We propose a framework for implementing existing minimization algorithms over $\X$ to directly solve \eqref{Prob_Ori}. Moreover, we show that we can construct the constraint dissolving mapping $\A$ directly from the expression of $c(x)$ for various closed convex sets $\X$ in Section \ref{Subsection_design_A}.

    \subsection{Theoretical analysis through constraint dissolving approach}
    \label{Subsection_implementation_framework}

    In this subsection, we demonstrate how existing minimization algorithms over $\X$ can be directly applied to solve \eqref{Prob_Ori}, with theoretical guarantees derived from established results. Our methodology involves two key stages. In the first stage, we transfer \eqref{Prob_Ori} into its corresponding constraint dissolving problem \eqref{Prob_Pen}. Then in the second stage, we solve \eqref{Prob_Pen} by existing algorithms for minimization over $\X$, such as the proximal gradient methods \cite{jorge2006numerical,beck2003mirror,li2015accelerated,davis2020stochastic,tang2024optimization},  quasi-Newton methods \cite{byrd1995limited,zhu1997LBFGSB,kimiaei2022lmbopt}, second-order methods \cite{nash1984newton,lin1999newton,conn2000trust,yuan2015recent}, etc.

    Let $\{\xk\}$ denote the sequence generated by applying the chosen algorithm to \eqref{Prob_Pen}, and $\overline{\Omega} := \cup_{x \in \K} \Omegax{x}$. Under the assumption that the iterates stay within a compact set $\Z \subseteq \Rn$ satisfying $\{\xk\} \subset \overline{\Omega}  \cap \Z$, the compactness of $\Z$ ensures that $\sup_{x \in \K \cap \Z} \beta_x < +\infty$. It is important to note that such $\Z$ exists when $\X$ is a compact subset of $\Rn$ (i.e., Assumption \ref{Assumption_error_bound}(1) holds). 
    Then by choosing $\beta > \sup_{x \in \K \cap \Z} \beta_x < +\infty$ in \eqref{Prob_Pen}, we present the following framework on establishing the convergence properties of the chosen algorithm through the equivalences between \eqref{Prob_Ori} and \eqref{Prob_Pen}. 
    \begin{itemize}
        \item {\bf Global convergence:} Theorem \ref{Theo_equivalence_feabile} and Theorem \ref{Theo_equivalence} demonstrate that \eqref{Prob_Ori} and \eqref{Prob_Pen} share the same first-order stationary points within $\overline{\Omega} \cap \Z$. Consequently, any cluster point $x^*$ of $\{\xk\}$ that is a first-order stationary point of \eqref{Prob_Pen} automatically satisfies the first-order optimality condition for \eqref{Prob_Ori}.
        \item {\bf Local convergence rate:} As established in Theorem \ref{Theo_equivalence_SOSC} and Theorem \ref{Theo_equivalence_strong_SOSC}, the SOSC points (or strong SOSC points) of \eqref{Prob_Ori} are SOSC points (or strong SOSC points) of \eqref{Prob_Pen} within $\overline{\Omega} \cap \Z$.  More importantly, the convexity of $\X$  ensures the non-degeneracy of the constraint $x \in \X$. Additionally, Theorem \ref{Theo_equivalence_strict_complementarity} illustrates the equivalence between \eqref{Prob_Ori} and \eqref{Prob_Pen} in the aspect of strict complementarity condition. These results on the equivalence between \eqref{Prob_Ori} and \eqref{Prob_Pen} enable direct extension of existing results \cite{izmailov2012stabilized,wang2023strong,liang2024squared} on local convergence rate under second-order sufficient conditions to \eqref{Prob_Ori} via \eqref{Prob_Pen}. 
        \item {\bf Worst-case complexity:} Theorem \ref{Theo_equivalence_approx} demonstrates that any $\varepsilon$-stationary point of \eqref{Prob_Pen} constitutes a $2\varepsilon$-stationary point of \eqref{Prob_Ori}. Therefore, a wide range of existing worst-case complexity results for minimization algorithms over $\X$ can be directly extended to \eqref{Prob_Ori}. 
    \end{itemize}

    Clearly, the existence of the compact subset $\Z$ is important for the validity of the above framework. In the following, we provide easy-to-verify conditions for the existence of such $\Z$. 
    Under Assumption \ref{Assumption_error_bound}(1), we set $\mu_{D} := \inf_{y \in \X \setminus \overline{\Omega}} \norm{c(y)}^2$ if $\X \setminus \overline{\Omega} \neq \emptyset$, otherwise, we set $\mu_{D}:= 1$. Additionally, we set $M_{D} := \sup_{x, y \in \X} |f(\A(x)) - f(\A(y))|$.
    Let
    \begin{equation}
        \label{Eq_defin_bar_beta}
        \bar{\beta} := \max\left\{ \sup_{x \in \K} \beta_x, \frac{4M_{D}}{\mu_{D}}  \right\}.
    \end{equation}
    \begin{prop}
        \label{Prop_monotone}
        Suppose Assumption \ref{Assumption_f}, Assumption \ref{Assumption_A}, and Assumption \ref{Assumption_error_bound}(1) hold. For any given $x \in \K$, let $\beta \geq \bar{\beta}$ for \eqref{Prob_Pen}, then it holds that 
        \begin{equation}
            \{y \in \X: h(y) \leq h(x)\} \subseteq \overline{\Omega}.
        \end{equation}
    \end{prop}
    \begin{proof}
        For any $y \in \X \setminus \overline{\Omega}$, it holds that 
        \begin{equation}
            h(y) - h(x) = f(\A(y)) - f(\A(x)) + \frac{\beta}{2} \norm{c(y)}^2 \geq -M_{D} + \frac{\beta \mu_D}{2} > 0.
        \end{equation}
        Therefore, we can conclude that $ \{y \in \X: h(y) \leq h(x)\} \subseteq \overline{\Omega}$. This completes the proof. 
    \end{proof}
    Therefore, under Assumption \ref{Assumption_error_bound}(1) with a sufficiently large but finite $\beta$ for \eqref{Prob_Pen}, by choosing any monotone algorithm or algorithm that employ non-monotone line search techniques \cite{grippo1986nonmonotone}, it is easy to guarantee that $h(x_{k+1}) \leq h(\xk)$ holds for any $k\geq 0$, hence the sequence of iterates $\{\xk\}$ are restricted within $\overline{\Omega}$. 

    In the following, we present an illustrative example on how to extend existing results for minimization over $\X$ to solving \eqref{Prob_Ori} via \eqref{Prob_Pen}. Specifically, we consider applying the projected gradient method \cite{lan2024projected} to solve \eqref{Prob_Pen}.
    The detailed algorithm is presented as follows. 
    \begin{algorithm}[htbp]
	\begin{algorithmic}[1]   
		\Require Input data: functions $f$, closed convex set $\X$, stepsizes $\{\eta_k\}$. 
            \State Choose $\beta \geq \bar{\beta}$ according to \eqref{Eq_defin_bar_beta} and construct the objective function $h$ for  \eqref{Prob_Pen}.
		\State Choose initial guess $x_0 \in \K$, set $k=0$.
		\While{not terminated}
		\State $\xkp = \Pi_{\X}(\xk - \eta_k \nabla h(\xk))$;
		\State $k = k+1$.
		\EndWhile
		\State Return $\xk$.
	\end{algorithmic}  
	\caption{Projected gradient method for \ref{Prob_Pen}.}  
	\label{Alg:PG}
\end{algorithm}

Under Assumption \ref{Assumption_error_bound}(1), from the compactness of $\X$, locally Lipschitz smoothness of $f$, $\A$ and $c$, we have that $\nabla h$ is Lipschitz continuous over $\X$. 
Now we choose constants $L_{h,u} \geq 0$ and $L_{h,l}\geq 0$ such that 
\begin{equation}
    -\frac{L_{h, l}}{2} \norm{x-y}^2 \leq h(x) - h(y) - \inner{\nabla h(y), x - y} \leq \frac{L_{h, u}}{2} \norm{x-y}^2,\quad \forall x, y\in\X.
\end{equation}
Moreover,  we denote $L_{\X} := \sup_{x, y \in \X} \norm{x-y}$, and $M_h := \sup_{x \in \X} \norm{\nabla h(x)}$. Then we have the following theorem demonstrating the convergence properties of Algorithm \ref{Alg:PG} from \cite[Theorem 1]{lan2024projected}. 
\begin{theo}
    \label{Theo_PG_Convergence}
    Suppose Assumption \ref{Assumption_f}, Assumption \ref{Assumption_A}, Assumption \ref{Assumption_Ax_x_diff}, and Assumption \ref{Assumption_error_bound}(1) hold, and we choose 
    \begin{equation*}
        \beta \geq \max\left\{ \bar{\beta}, ~\sup_{x \in \K} \frac{8(\Mxjc \Mxf + \Lxres)(\Mxa + \Mxr + 1)}{\sigmaxc}  \right\}
    \end{equation*}
    for \eqref{Prob_Pen}. Let the sequence $\{\xk\}$ be generated by Algorithm \ref{Alg:PG} with $\eta_k = \frac{1}{L_{h,u}}$. Then for any $K>1$, there exists $k \leq K$ such that $\xk$ is a $\iota_K$-first-order stationary point of \eqref{Prob_Ori}, where 
    \begin{equation}
        \iota_K = 4\left(\frac{\max\{L_{h, u}, L_{h, l}\} }{L_{h, u}} + 1\right)\sqrt{\frac{2L_{h,u}^2 L_{\X}^2}{K(K-1)} + \frac{2L_{h,u}L_{h, l} L_{\X}^2}{K-1}}.
      \end{equation}
\end{theo}
\begin{proof}
    For the sequence $\{\xk\}$, it directly follows from \cite[Theorem 1]{lan2024projected} that 
    \begin{equation*}
        \inf_{i\leq K-1} \norm{\frac{1}{\eta_{i}} (x_{i+1} - x_i) }
        \leq \sqrt{\frac{2L_{h,u}^2 L_{\X}^2}{K(K-1)} + \frac{2L_{h,u}L_{h, l} L_{\X}^2}{K-1}}.
    \end{equation*}
    Then we can conclude that there exists $k \leq K$ such that 
    \begin{equation*}
        \mathrm{dist}\left(0, \nabla h(\xk) + \NX(\xkp) \right) \leq \sqrt{\frac{2L_{h,u}^2 L_{\X}^2}{K(K-1)} + \frac{2L_{h,u}L_{h, l} L_{\X}^2}{K-1}}.
    \end{equation*}
    As a result, it holds from the continuity of $\nabla h$ that 
    \begin{equation*}
        \begin{aligned}
            &\mathrm{dist}\left(0,\nabla h(\xkp) + \NX(\xkp) \right) \\
            \leq{}&  \max\{L_{h, u}, L_{h, l}\} \norm{\xkp - \xk} +  \sqrt{\frac{2L_{h,u}^2 L_{\X}^2}{K(K-1)} + \frac{2L_{h,u}L_{h, l} L_{\X}^2}{K-1}} \\
            \leq{}& 2\left(\frac{\max\{L_{h, u}, L_{h, l}\} }{L_{h, u}} + 1\right)\sqrt{\frac{2L_{h,u}^2 L_{\X}^2}{K(K-1)} + \frac{2L_{h,u}L_{h, l} L_{\X}^2}{K-1}} = \frac{1}{2} \iota_K. 
        \end{aligned}
    \end{equation*}
    This illustrate that $\xk$ is a $\frac{1}{2} \iota_K$-first-order stationary point of \eqref{Prob_Pen}. Together with Theorem \ref{Theo_equivalence_approx}, we can conclude that $\xk$ is an $\iota_K$-first-order stationary point of \eqref{Prob_Ori}. This completes the proof. 
\end{proof}

    Theorem \ref{Theo_PG_Convergence} demonstrates that directly applying the proximal gradient method to solve \eqref{Prob_Ori} through \eqref{Prob_Pen} achieves an $\ca{O}(\varepsilon^{-2})$ worst-case complexity. This result matches existing complexity bounds established for augmented Lagrangian methods \cite{cartis2019optimality,xie2021complexity} and sequential quadratic programming methods \cite{curtis2024worst} under the nondegeneracy conditions of \eqref{Prob_Ori}. Therefore, we can conclude that solving \eqref{Prob_Ori} through \eqref{Prob_Pen} admits direct implementations of various existing methods for minimization over $\X$, with convergence properties that directly follow from existing established theoretical results.

    \subsection{Designing constraint dissolving mappings}
    \label{Subsection_design_A}

    In this subsection, we discuss how to design constraint dissolving mappings for the constraints in \eqref{Prob_Ori}. It is worth mentioning that the formulation of a constraint dissolving mapping only depend on $c$ and $\X$, and hence it is decoupled from the objective function $f$. For the special case where $\X = \Rn$, some prior works \cite{xiao2023dissolving,hu2022constraint,hu2022improved,xiao2022cdopt} have discussed how to design constraint dissolving mappings for \eqref{Prob_Ori}. However, no existing work has discussed how to design a constraint dissolving mapping 
    $\A$ with a general closed convex set $\X$.

    \begin{table}[tb]
    \centering
    \small
    \begin{tabular}{p{4cm}|p{4.5cm}|p{6cm}}
        \hline
        \textbf{Name of constraints} & \textbf{Formulation of $\X$} & \textbf{Possible choices of projective mapping} \\ \hline
        Box constraints &
        $\{x \in \Rn: l \leq x \leq u\}$ &
        $Q(x) = \mathrm{Diag}(q_1(x), \ldots, q_n(x))$ \newline where $\{q_i\}$ are Lipschitz smooth, $q_i(l_i) = q_i(u_i) = 0$, and $q_i(s) > 0$ for $s \in (l_i, u_i)$ \\ \hline
        
        Norm ball &
        $\{x \in \Rn: \norm{x} \leq u\}$ &
        $Q(x) = I_n - \frac{xx\tp}{u^2}$ \\ \hline
        
        Norm ball ($\ell_q$ norm with $q>1$) &
        $\{x \in \Rn: \norm{x}_q \leq u\}$ &
        $Q(x) = I_n - \frac{2}{u^q} \Phi\left(\mathrm{sign}(x) \odot |x|^{q-1} x\tp \right) + \frac{1}{u^{2q}}
        \sum_{i=1}^n |x_i|^{2q-2} xx\tp$ \\ \hline
        
        Probability simplex &
        $\{x \in \Rn: x \geq 0, \norm{x}_1 = 1\}$ &
        $Q(x) = \left(\text{Diag}(x) - xx\tp \right)^2$ \\ \hline
        
        General linear constraints &
        $\{x \in \Rn: A\tp x \leq b\}$ &
        $Q(x) = I_n + (A^{\dagger})\tp ((\max\{b - A\tp x, 0\})^{\odot 2} - \mathbbm{1}_n) A^{\dagger}$ \\ \hline
        
        Second-order cone &
        $\{(x, y) \in \Rn \times \bb{R}: \norm{x} \leq y\}$ &
        $Q(x) = (\norm{x}^2 + y^2)I_{n+1} - \exp(\norm{x}^2 - y^2)\begin{bmatrix} xx\tp & -xy \\ -yx\tp & y^2 \end{bmatrix}$ \\ \hline
        
        Spectral constraints &
        $\{X \in \bb{R}^{m\times s} : \norm{X}_2 \leq 1\}$ &
        $Q(X): Y \mapsto Y - X \Phi(X\tp Y)$ \\ \hline
        
        PSD cone &
        $\{X \in \bb{R}^{s\times s} : X \succeq 0\}$ &
        $Q(X): Y \mapsto XYX$ \\ \hline
        
        PSD matrices with spectral constraints &
        $\{X \in \bb{R}^{s\times s} : X \succeq 0, \norm{X}_2 \leq 1\}$ &
        $Q(X): Y \mapsto XY X - X^2 Y X^2$ \\ \hline
    \end{tabular}
    \caption{Some constraints and their corresponding projective mappings. Here $\Phi$ is the symmetrization of a square matrix, defined as $\Phi(M):= \frac{M+M\tp}{2}$. }
    \label{Table_Q_mapping}
\end{table}

    We first present a general formulation for constructing the constraint 
    dissolving mapping $\A$. Define the mapping $Q: \Rn \to \ca{S}^n$ (called a projective mapping), which only depends on the choice of $\X$, and satisfies the following conditions. 
    \begin{assumpt}
        \label{Assumption_Q}
        \begin{enumerate}
            \item $Q$ is continuously differentiable over $\X$;
            \item $Q(x)$ is positive semi-definite for any $x \in \X$;
            \item $\mathrm{null}(Q(x)) = \mathrm{range}(\ca{N}_{\X}(x))$ for all $x \in \X$. 
        \end{enumerate}
    \end{assumpt}
    Then based on the projective mapping $Q$, we consider the following mapping $\A_{Q}: \X \to \Rn$: 
    \begin{equation}
        \label{Eq_AQ_general}
        \A_Q(x) := x - Q(x) \Jc(x) \left(\Jc(x)\tp Q(x) \Jc(x) + \sigma \norm{c(x)}^2 I_p\right)^{\dagger} c(x), 
    \end{equation}
    where $\sigma> 0$ is a given positive constant.

    Here we make some comments on the conditions in Assumption \ref{Assumption_Q}. Assumption \ref{Assumption_Q}(2) serves as the regularization condition to ensure that $\Jc(x)\tp Q(x) \Jc(x) + \sigma \norm{c(x)}^2 I_p$ is non-singular when $r=p$, and hence $\A_Q$ is well-defined over $\X$ (see Lemma \ref{Le_JcQJc_nondiminishing} for detail). Moreover, Assumption \ref{Assumption_Q}(1) assumes the differentiability of $Q$, which guarantees that $\A_Q$ is locally Lipschitz smooth over $\X$. Furthermore, as illustrated later in Lemma \ref{Le_Null_JAQ}, the null space of $\nabla\A_Q(x)$ coincides with the null space of $Q(x)$. As a result, Assumption \ref{Assumption_Q}(3) is introduced to enforce the validity of Assumption \ref{Assumption_A}(3) for $\A_{Q}$.

    Next, we aim to show that for any mapping $Q$ that satisfies Assumption \ref{Assumption_Q}, the mapping $\A_{Q}$ satisfies all the conditions in Assumption \ref{Assumption_A}. We begin our theoretical analysis with the following lemma showing 
    that $\A_{Q}$ is well defined.
    \begin{lem}
        \label{Le_JcQJc_nondiminishing}
        Suppose Assumption \ref{Assumption_f} holds with $r = p$.
        For any given projective mapping $Q: \X\to \ca{S}^n$
        satisfying Assumption \ref{Assumption_Q}, we have that $\Jc(x)\tp Q(x) \Jc(x) \succ 0$ holds for all $x \in \K$. 
    \end{lem}
    \begin{proof}
        We prove this lemma by contradiction. Suppose there exists $x \in \K$ such that $\Jc(x)\tp Q(x) \Jc(x)$ is not full-rank. Then there exists $w_x \in \Rp$ such that 
        \begin{equation*}
           0 \neq \Jc(x)w_x \in \mathrm{null}(Q(x)) = \mathrm{range}(\ca{N}_{\X}(x)).
        \end{equation*}
        
        By Assumption~\ref{Assumption_f}(2b) with $r=p$, we know that 
        $w_x = \Jc(x)^\top d$ for some $d\in\mathrm{lin}(\mathcal{T}_\X(x)).$ Thus
        \begin{eqnarray*}
        0 = \inner{d,\Jc(x) w_x}  = \inner{\Jc(x)^\top d, w_x} 
        =\inner{w_x,w_x},
        \end{eqnarray*}
        which implies that $w_x =0$ and hence $\Jc(x)w_x=0$.   
        But this contradicts that $\Jc(x)w_x\neq 0$.
               Therefore, we can conclude that $\Jc(x)\tp Q(x) \Jc(x)$ is full-rank for any $x \in \K$. This completes the proof.   
    \end{proof}

    Lemma \ref{Le_JcQJc_nondiminishing} illustrates that for all $x \in \X$, we have $\Jc(x)\tp Q(x) \Jc(x) + \sigma \norm{c(x)}^2 I_p \succ 0$ when $r=p$. Then we can define the mapping $T: \Rn \to \bb{R}^{n\times n}$ 
    \begin{equation}
        T_Q(x) := Q(x) \Jc(x)\left( \Jc(x)\tp Q(x) \Jc(x) + \sigma \norm{c(x)}^2 I_p \right)^{-1}. 
    \end{equation}
    We can conclude from Lemma \ref{Le_JcQJc_nondiminishing} that the mapping $T_Q$ is well-defined over $\X$. As a result, $\A_{Q}(x) = x-T_Q(x)c(x)$ is well defined for all $x \in \X$.

    Furthermore, we have the following lemma characterizing the null space of $\nabla {\A_Q}(x) - I_n$. 
    \begin{lem}
        \label{Le_Null_JAQ}
        Suppose Assumption \ref{Assumption_f} hold with $r = p$. Then for any given mapping $Q$ satisfying Assumption \ref{Assumption_Q}, any $x \in \K$ and any $d \in \mathrm{null}(Q(x))$, we have
        \begin{equation*}
            (\nabla {\A_Q}(x) - I_p)d = 0. 
        \end{equation*}
    \end{lem}
    \begin{proof}
        From the expression of $\A_{Q}$ in \eqref{Eq_AQ_general}, it is easy to verify that for any $x \in \K$, we have
        \begin{equation*}
            \Ja_Q(x) - I_p = - \Jc(x) \left(\Jc(x)\tp Q(x) \Jc(x) + \sigma \norm{c(x)}^2 I_p\right)^{-1} \Jc(x)\tp  Q(x). 
        \end{equation*}
        Therefore, for any $d \in \mathrm{null}(Q(x))$, it is clear that 
        $(\nabla{\A_Q}(x) - I_p ) d =0$.
        This completes the proof. 
    \end{proof}

    The following proposition illustrates that $\A_{Q}$ satisfies all the conditions in Assumption \ref{Assumption_A} and Assumption \ref{Assumption_Ax_x_diff}, hence can be regarded as a constraint dissolving mapping for \eqref{Prob_Pen}. 
    \begin{prop}
        \label{Prop_Construct_A_from_Q}
        Suppose Assumption \ref{Assumption_f} holds with $r=p$. 
        For any given mapping $Q: \X\to \ca{S}^n$ satisfying \ref{Assumption_Q},
         the mapping $\A_Q$ defined in \eqref{Eq_AQ_general} satisfies all the conditions in Assumption \ref{Assumption_A} and Assumption \ref{Assumption_Ax_x_diff}. 
    \end{prop}
    \begin{proof}
        As demonstrated in Lemma \ref{Le_JcQJc_nondiminishing}, the mapping $\A$ is well-defined in $\X$. Moreover, from the continuity of $c$, $\Jc$, and $Q$, we can conclude that $\A_Q$ is locally Lipschitz continuous over $\X$.

        We can easily see check that $\A_Q(x) = x$ holds for all $x \in \K$. Furthermore, for $y\in \X$,
        \begin{equation*}
            \nabla {\A_Q}(y) = I_n - \Jc(y) \left( \Jc(y)\tp Q(y) \Jc(y) + \sigma \norm{c(y)}^2 I_p \right)^{-1} \Jc(y)\tp  Q(y) - D_{T_Q}(y)[c(y)] 
        \end{equation*}
        where $D_{T_Q}(y)$ denotes the Jacobian of $T_Q$ at $y$.
        Then it is easy to verify that for any $x \in \K$, it holds that 
        \begin{equation*}
            \nabla {\A_Q}(x) = I_n - \Jc(x) \left( \Jc(x)\tp Q(x) \Jc(x) \right)^{-1} \Jc(x)\tp  Q(x).
        \end{equation*}
        As a result, for any $x \in \K$, we have
        \begin{equation*}
            \nabla {\A_Q}(x) \Jc(x) = \Jc(x) - \Jc(x) \left( \Jc(x)\tp Q(x) \Jc(x)\right)^{-1}  \Jc(x)\tp Q(x) \Jc(x) = \Jc(x) - \Jc(x) = 0. 
        \end{equation*}
        In addition, let $R_Q(y) = -D_{T_Q}(y)[c(y)] $ for $y\in\X$. Then it holds that $R_Q(x) = 0$ holds for any $x \in \K$. Moreover, for any $y \in \X$, it holds that 
        \begin{equation*}
             (\nabla{\A_Q}(y) - R_{Q}(y) - I_n) d = 0, \quad \forall\;
             d \in \NX(y). 
        \end{equation*}
        Therefore, we conclude the validity of Assumption \ref{Assumption_A}(3). 

        Furthermore, notice that 
        \begin{equation*}
            \begin{aligned}
                &\norm{\A_{Q}(y)-y} = \norm{Q(y) \Jc(y) \left(\Jc(y)\tp Q(y) \Jc(y) + \sigma \norm{c(y)}^2 I_p\right)^{\dagger} c(y)} \\
                \leq{}& \norm{Q(y) \Jc(y) \left(\Jc(y)\tp Q(y) \Jc(y) + \sigma \norm{c(y)}^2 I_p\right)^{\dagger}} \norm{c(y)}.
            \end{aligned}
        \end{equation*}
        Then Assumption \ref{Assumption_Ax_x_diff} holds with $\Lxres = \norm{Q(y) \Jc(y) \left(\Jc(y)\tp Q(y) \Jc(y) + \sigma \norm{c(y)}^2 I_p\right)^{\dagger}}$. 
        This completes the proof. 
    \end{proof}

    In addition, for the case where $r < p$, we present the following proposition illustrating that the mapping $\A_{Q}$ is well defined. The proof of Proposition \ref{Prop_linear_Construct_A_from_Q} directly follows from direct calculations, hence is omitted for simplicity. 
    \begin{prop}
        \label{Prop_linear_Construct_A_from_Q}
        Suppose Assumption \ref{Assumption_f} holds. 
         For any given mapping $Q: \X\to \ca{S}^n$ satisfying Assumption \ref{Assumption_Q}, if $c: \Y \to \Rp$ is an affine mapping, then the mapping $\A_Q$ defined in \eqref{Eq_AQ_general} satisfies all the conditions in Assumption \ref{Assumption_A} and Assumption \ref{Assumption_Ax_x_diff}. 
    \end{prop}

    As demonstrated in our proposed scheme  \eqref{Eq_AQ_general},  constructing the constraint dissolving mapping in \eqref{Eq_AQ_general} requires the evaluations of constraints $c$, its Jacobian $\Jc$, and the projective mapping $Q$. In this part, we discuss how to choose a projective mapping $Q$ for any given closed convex subset $\X$.

    For a variety of commonly encountered $\X$, we present some possible choices of the projective mappings $Q$ in Table \ref{Table_Q_mapping}. It can be easily verified that all the projective mappings $Q$ in Table \ref{Table_Q_mapping} satisfy Assumption \ref{Assumption_Q}, and we omit the proofs for simplicity. As a result, based on these formulations of $Q$, Proposition \ref{Prop_Construct_A_from_Q} illustrates that the corresponding constraint dissolving mappings in the form of \eqref{Eq_AQ_general} satisfy Assumption \ref{Assumption_A}. 

    We end this section by presenting some possible choices of the constraint dissolving mappings $\A$ for some $\X$ with specific constraints $c(x)=0$ in Table \ref{Table_A_specific}. It is worth mentioning that all these choices of the constraint dissolving mappings only involve matrix-matrix multiplications and, hence, can be efficiently computed in practice.

\begin{table}[tb]
\centering
\begin{tabular}{l|l|l}
\hline
\textbf{Formulation of $\X$} & \textbf{Formulations of $c$} & \textbf{Constraint dissolving mapping} \\
\hline
$\{x \in \Rn: x \geq 0\}$ & $c(x) = x\tp Hx - 1$ with $H \in \ca{S}^n$ & $\A(x) = x - \frac{1}{2} x(x\tp H x - 1)$ \\
\hline
$\{x \in \Rn: x \geq 0\}$ & $c(x) = \norm{x}_q^q - 1$ with $q > 1$ & $\A(x) = x/(1 + \frac{1}{q}(\norm{x}_q^q - 1))$ \\
\hline
$\{X \in \bb{R}^{s\times s} : X \succeq 0\}$ & $c(X) = \mathrm{Diag}(X) - I_n$ & $\A(X) = \Phi( X(2I_n -  \mathrm{Diag}(X)) )$ \\
\hline
$\{X \in \bb{R}^{m\times s}: X \geq 0\}$ & $c(X) = \mathrm{Diag}(X\tp X - I_p)$ & $\A(X) = X - \frac{1}{2} X \mathrm{Diag}(X\tp X - I_p)$ \\
\hline
\end{tabular}
\caption{Possible constraint dissolving mappings for some specific constraint functions $c(\cdot)$.}
\label{Table_A_specific}
\end{table}

    \section{Numerical Experiments} 
    In this section, we present numerical experiments to evaluate the efficiency of our constraint dissolving approach, where the constrained optimization problem \eqref{Prob_Ori} is solved by applying solvers for minimization over $\X$ to  \eqref{Prob_Pen}.  All experiments are conducted in Python 3.12.2 on a MacOS server equipped with an M3 Max chip and 64 GB of RAM. 

    For solving the constrained optimization problem \eqref{Prob_Ori}, we choose solvers from the SciPy package (available at \href{https://scipy.org}{https://scipy.org}). Specifically, we employ the Sequential Least Squares Programming (SLSQP) and Trust Region Constrained Algorithm (TRCON), both of which are designed for solving general constrained optimization problems of the form \eqref{Prob_Ori}. Moreover, for solving \eqref{Prob_Pen}, we employ the Limited-memory BFGS solver (L-BFGS-B) \cite{zhu1997LBFGSB} from the SciPy package, which is capable of solving \eqref{Prob_Pen} when $\X$ is the box constraint. Furthermore, we independently implement the projected gradient method with a Barzilai-Borwein stepsize and non-monotone line search (PG-BB) \cite{birgin2000nonmonotone} to solve \eqref{Prob_Pen} with a general $\X$.

    To ensure consistency across algorithms with varying stopping criteria, we evaluate the stationarity of the output $x$ using the normalized projected gradient norm $\norm{\Pi_{\X}(x-\nabla h(x))}/(1+\norm{x})$ of the penalty function in \eqref{Prob_Pen}. Additionally, we use $\norm{c(\Pi_\X(x))}$ to measure the violation of the feasibility of the output $x$.

    \subsection{Nonnegative sparse principal component analysis}
    In this subsection, we evaluate the numerical performance of our proposed constraint dissolving approach on the nonnegative sparse principal component analysis (PCA) problem \cite{zass2006nonnegative}, which takes the following formulation, 
    \begin{equation}\label{Example_NPCA}
        \begin{aligned}
            \min_{x \in \Rn} \quad &-\frac{1}{2}\norm{B\tp x}^2 + \rho\cdot {\bf 1}^\top x\\
            \text{s. t.} \quad & c(x):= \norm{x}^2 - 1 =0, \quad 
            x \in \X = \mathbb{R}^n_+,
        \end{aligned}
    \end{equation}
    where $B \in \bb{R}^{n\times q}$ is the data matrix and $\rho\geq 0$ is the penalty parameter. 
    
    In our numerical experiments, the data matrix $B$ is randomly generated from the standard normal distribution and normalized such that $\norm{B}_2 = q$. The penalty parameter $\beta$ is set to $100$, and the tolerance for feasibility and stationarity violations is set to $10^{-6}$. In each test instance, all the compared algorithms start from the same initial point $|x|/\norm{x}$, where $x \in \mathbb{R}^n$ is randomly sampled from the standard normal distribution.

    \begin{table}[tb]
		\begin{center}
			\footnotesize
			\begin{minipage}{\textwidth}
				\caption{A comparison between solvers for \eqref{Prob_Ori} and those for \eqref{Prob_Pen} on nonnegative sparse PCA \eqref{Example_NPCA}.}
				\label{Table_NPCA}
				\begin{tabular*}{\textwidth}{c@{\extracolsep{\fill}}ccccccc@{\extracolsep{\fill}}}
					\toprule \midrule
					Test instances& Solvers &
					Function value &Feasibility & Stationarity 
					& CPU time (s) \\
                    \hline
                        $n=100$ & PG-BB & -4.6454325e+01 & 3.08e-08 & 1.92e-07 & 0.03 \\
                        $q=50$ & L-BFGS-B & -4.6454324e+01 & 1.45e-09 & 5.55e-08 & 0.02 \\
                        $\rho=0.00$ & SLSQP & -4.6454324e+01 & 6.57e-10 & 6.08e-06 & 0.12 \\
                            & TRCON & -4.6454265e+01 & 1.45e-13 & 1.67e-05 & 4.50 \\
                        \hline
                        $n=200$ & PG-BB & -3.5536922e+01 & 1.42e-08 & 8.86e-08 & 0.04 \\
                        $q=50$ & L-BFGS-B & -3.5536921e+01 & 5.45e-09 & 5.88e-08 & 0.03 \\
                        $\rho=0.00$ & SLSQP & -3.5536981e+01 & 1.68e-06 & 7.75e-06 & 1.27 \\
                            & TRCON & -3.5536801e+01 & 1.30e-12 & 9.74e-06 & 11.50 \\
                        \hline
                        $n=500$ & PG-BB & -2.6016263e+01 & 1.67e-09 & 9.82e-07 & 0.07 \\
                        $q=50$ & L-BFGS-B & -2.6016263e+01 & 3.61e-09 & 1.18e-07 & 0.05 \\
                         $\rho=0.00$ & SLSQP & -2.6016282e+01 & 7.52e-07 & 5.02e-05 & 23.40 \\
                                 & TRCON & -2.6014736e+01 & 1.42e-11 & 2.48e-04 & 56.98 \\
                        \hline
                        $n=100$ & PG-BB & -4.5837277e+01 & 3.58e-11 & 7.85e-09 & 0.04 \\
                        $q=50$ & L-BFGS-B & -4.5837277e+01 & 4.47e-09 & 4.74e-08 & 0.02 \\
                        $\rho=0.10$ & SLSQP & -4.5837277e+01 & 6.97e-10 & 8.50e-06 & 0.12 \\
                         & TRCON & -4.5837217e+01 & 2.99e-13 & 1.55e-05 & 5.07 \\
                        \hline
                        $n=200$ & PG-BB & -3.4705156e+01 & 9.66e-11 & 1.17e-07 & 0.04 \\
                        $q=50$ & L-BFGS-B & -3.4705156e+01 & 8.36e-09 & 6.16e-08 & 0.03 \\
                        $\rho=0.10$ & SLSQP & -3.4705156e+01 & 1.37e-11 & 1.89e-06 & 1.27 \\
                                     & TRCON & -3.4705031e+01 & 1.74e-12 & 1.43e-05 & 10.31 \\
                        \hline
                        $n=500$ & PG-BB & -2.4741473e+01 & 6.29e-10 & 8.04e-08 & 0.07 \\
                        $q=50$ & L-BFGS-B & -2.4741473e+01 & 7.30e-09 & 9.07e-08 & 0.05 \\
                        $\rho=0.10$ & SLSQP & -2.4741474e+01 & 3.43e-08 & 9.92e-05 & 23.23 \\
                                    & TRCON & -2.4739895e+01 & 1.35e-11 & 2.29e-04 & 57.23 \\
                    \hline
					\bottomrule
				\end{tabular*}
			\end{minipage}
		\end{center}
	\end{table}

    The results of our numerical experiments are presented in Table \ref{Table_NPCA}. From these numerical results, we can conclude that solving \eqref{Prob_Ori} through PG-BB and L-BFGS-B on \eqref{Prob_Pen} is generally faster than directly applying SLSQP and TRCON to solve \eqref{Prob_Ori}. Notably, in the third and sixth examples, PG-BB and L-BFGS-B are over $100$ times faster than SLSQP and TRCON. These findings highlight the effectiveness of our constraint dissolving approach, which reformulates problems with complex constraints into simpler problems with convex constraints, and thus enabling the direct implementation of efficient gradient-based methods. It is important to note that PG-BB and L-BFGS-B cannot be directly applied to \eqref{Prob_Ori}, as its feasible set $\K$ is typically nonconvex due to the nonlinear constraint $c(x) = 0$.

    \subsection{Nonconvex quadratic programming with ball constraints}
    In this subsection, we test the numerical performance of the constraint dissolving approach on the following nonconvex quadratic programming (QP) with ball constraints,
    \begin{equation}\label{Example_QPB}
        \begin{aligned}
            \min_{x \in \Rn} \quad &\frac{1}{2}\inner{x, Q x} + \inner{q, x}\\
            \text{s. t.} \quad & \norm{x-d}^2=1, \quad 
            x\in \X := \{ z\in\Rn : \norm{z} \leq 1\}.
        \end{aligned}
    \end{equation}
    Here $Q \in \bb{R}^{n \times n}$ is a symmetric and indefinite data matrix, $q\in \Rn$, $d=[1/2,0,\ldots,0]^\top \in \Rn.$ The nonconvex QP problem \eqref{Example_QPB} is originally studied in \cite{burer2024slightly}, where both of the quadratic constraints are inequality constraints. Here, we reformulate one of these constraints as an equality constraint to align its formulation with our model \eqref{Prob_Ori}. Such a reformulation is justified, as one of the inequality constraints is guaranteed to be active at the optimal solution when $Q$ is indefinite. 

    Since the L-BFGS-B solver cannot handle ball constraints, we solve the constraint dissolving counterpart of \eqref{Example_QPB} by PG-BB, and compare its performance with SLSQP and TRCON, which are employed to solve \eqref{Example_QPB} directly. The data matrix $Q$ is generated as $Q = -L/\norm{L}\ff$, where $L$ is the Laplacian matrix of a random graph with edge density chosen as $0.5$. The vector $q \in \Rn$  is chosen as $q = q'/ \norm{q'}$, where $q'$
    is sampled from a uniform distribution. 
    The initial point for all algorithms is set to $x_0 = x'_0/\norm{x'_0}$, where $x'_0$ is randomly generated from the standard normal distribution. For all test instances, we choose the penalty parameter $\beta$ to be 10 and the stationarity tolerance to be $10^{-6}$. Notably, due to differences in the stopping criteria different among solvers from SciPy, the stationarity of the output of SLSQP may fall below $10^{-6}$. To ensure a fair comparison, we continue running PG-BB until the accuracy of its solution surpasses that of the SLSQP output.

    \begin{table}[tb]
		\begin{center}
			\footnotesize
			\begin{minipage}{\textwidth}
				\caption{A comparison between solvers for \eqref{Prob_Ori} and those for \eqref{Prob_Pen} on nonconvex QP with ball constraints \eqref{Example_QPB}.}
				\label{Table_QPB}
				\begin{tabular*}{\textwidth}{c@{\extracolsep{\fill}}ccccccc@{\extracolsep{\fill}}}
					\toprule \midrule
					Test instances& Solvers &
					Function value &Feasibility & Stationarity 
					& CPU time (s) \\
                    \hline
                        \multirow{4}{*}{$n=100$} 
                        & PG-BB & -9.5779118e-01 & 6.98e-10 & 8.21e-09 & 0.03 \\ 
                         & SLSQP & -9.5779118e-01 & 8.01e-10 & 7.73e-07 & 0.01 \\ 
                         & TRCON & -9.5779118e-01 & 1.10e-06 & 4.82e-06 & 0.03 \\
                         \hline
                        \multirow{4}{*}{$n=200$} 
                         & PG-BB & -9.6301560e-01 & 4.00e-10 & 5.97e-09 & 0.02 \\
                        & SLSQP & -9.6301560e-01 & 1.60e-10 & 9.90e-08 & 0.03 \\
                         & TRCON & -9.6301560e-01 & 5.11e-06 & 2.12e-05 & 0.05 \\
                        \hline
                        \multirow{4}{*}{$n=500$} 
                            & PG-BB & -9.6802779e-01 & 4.43e-10 & 3.63e-09 & 0.03 \\
                         & SLSQP & -9.6802779e-01 & 4.16e-10 & 4.15e-08 & 0.45 \\
                          & TRCON & -9.6802779e-01 & 4.54e-06 & 1.75e-05 & 0.12 \\
                        \hline
                        \multirow{4}{*}{$n=1000$} 
                          & PG-BB & -9.6541928e-01 & 2.74e-13 & 1.14e-12 & 0.01 \\
                          & SLSQP & -9.6541928e-01 & 9.02e-10 & 5.20e-08 & 3.48 \\
                           & TRCON & -9.6541928e-01 & 2.34e-05 & 9.19e-05 & 0.46 \\
                        \hline
                        \multirow{4}{*}{$n=2000$} 
                          & PG-BB & -9.6991471e-01 & 6.32e-11 & 1.33e-09 & 0.04 \\ 
                           & SLSQP & -9.6991471e-01 & 8.09e-10 & 4.60e-08 & 28.03 \\
                            & TRCON & -9.6991471e-01 & 2.17e-05 & 8.11e-05 & 1.04 \\
                        \hline
                        \multirow{4}{*}{$n=5000$} 
                            & PG-BB & -9.6950753e-01 & 5.62e-11 & 2.90e-10 & 0.28 \\ 
                           & SLSQP & -9.6950753e-01 & 4.67e-10 & 3.26e-08 & 451.33 \\ 
                             & TRCON & -9.6950753e-01 & 2.16e-05 & 8.08e-05 & 13.71 \\
                        \hline
					\bottomrule
				\end{tabular*}
			\end{minipage}
		\end{center}
	\end{table}

The numerical results are presented in Table \ref{Table_QPB}. We can observe that PG-BB outperforms SLSQP and TRCON in the aspect of CPU time for all but the first test instance, where PG-BB demonstrates comparable efficiency as TRCON. Specifically, for high-dimensional test instances,  PG-BB can be more than $1000$ times faster than SLSQP, as demonstrated in the last test instance. While the performance difference between PG-BB and TRCON is less significant, but TRCON fails to meet the required tolerance in some test instances. These numerical results demonstrate that PG-BB consistently achieves more accurate solutions in less CPU time, when compared to the SLSQP and TRCON solvers.

\subsection{Fair principal component analysis}
In this subsection, we test the efficiency of our developed constraint dissolving approach on the fair principal component analysis (PCA) problems \cite{samadi2018price}, which is introduced to impose fairness among different groups inside a dataset when performing principal component analysis. The fair PCA problem takes the following formulation,
\begin{equation}\label{Example_FPCA}
        \begin{aligned}
            \min_{P \in \bb{R}^{n\times d}} \quad &f_{\rm fpca}(P) = \max_{1\leq i\leq k}\left\{ \frac{-\inner{A^\top_i A_i,PP^\top}+\norm{\widehat{A_i}}^2}{m_i} \right\}\\
            \text{s. t.} \quad & P^\top P=I_d. 
        \end{aligned}
\end{equation}    
Here for any $i\in [k],$ $A_i\in \bb{R}^{m_i\times n}$ is a data matrix, $\widehat{A_i}$ is the best rank-$d$ approximation of $A_i$, obtained by truncated singular value decomposition (SVD). We reformulate \eqref{Example_FPCA} into the following equivalent problem, which fits into the formulation of \eqref{Prob_Ori},
\begin{equation}\label{Example_FPCA1}
        \begin{aligned}
            \min_{P \in \bb{R}^{n\times d},y\in \bb{R}^k,z\in \bb{R}} \quad & z \\
            \text{s. t.} \quad & \forall i\in [k], \frac{-\inner{A^\top_i A_i,PP^\top}+\norm{\widehat{A_i}}^2}{m_i}+y_i=z\\
            & \norm{P}_F^2=d, \quad \norm{P}_2\leq 1, \quad y\geq 0.
        \end{aligned}
\end{equation} 
In \eqref{Example_FPCA1}, the constraints $\norm{P}_F^2=d, \ \norm{P}_2\leq 1$ are equivalent to that $P^\top P=I_d$. Such a reformulation includes the convex conic constraint $\norm{P}_2 \leq 1$, which does not appear in the previous numerical examples. It is important to note that all the solvers in SciPy cannot handle such constraints. On the other hand, as the projection onto the convex set 
$\X=\{(P, y, z) \in \bb{R}^{n \times d} \times \mathbb{R}^k \times \bb{R}: \norm{P}_2 \leq 1, ~ y\geq 0\}$ is easy-to-compute, PG-BB can be directly implemented to solve \eqref{Example_FPCA1} through its constraint dissolving counterpart. 

In our numerical experiments, for any $i\in [k]$, we generate $A_i\in \bb{R}^{n\times n}$ from the standard normal distribution, and choose $d = 3$. For choosing the initial value $(P_0, y_0, z_0)$ for \eqref{Example_FPCA1}, we set $P_0 = \sqrt{d}R/\norm{R}_F$ with randomly generated $R$ from $\bb{R}^{n\times d}$. Moreover, $z_0$ is set to be $f_{\rm fpca}(P_0) + 1$, and $y_0$ is chosen sufficiently large to satisfy the equality constraints in \eqref{Example_FPCA1}. For each test instance, we test different penalty parameters $\beta\in \{0.1,1,10\}$ and select the one that performs best. The tolerance is set to a moderate value of $10^{-4}$.

    \begin{table}[tb]
		\begin{center}
			\footnotesize
			\begin{minipage}{\textwidth}
				\caption{Test results of PG-BB for solving the
                \eqref{Prob_Pen} formulation of the fair PCA problem \eqref{Example_FPCA1}.}
				\label{Table_FPCA}
				\begin{tabular*}{\textwidth}{c@{\extracolsep{\fill}}ccccccc@{\extracolsep{\fill}}}
					\toprule \midrule
					Test instances& Solvers &
					Function value &Feasibility & Stationarity 
					& CPU time (s) \\
                    \hline
                        $n=100,k=2$ 
                        & PG-BB & 2.9399515e+00 & 7.21e-05 & 5.84e-06 & 0.41 \\    
                        $n=100,k=5$
                         & PG-BB & 4.9882187e+00 & 2.31e-05 & 3.72e-05 & 1.51 \\ 
                        $n=100,k=10$
                         & PG-BB & 5.9220659e+00 & 7.91e-05 & 5.37e-05 & 4.52 \\
                         \hline
                         $n=500,k=2$
                         & PG-BB & 3.0785754e+00 & 2.14e-05 & 8.38e-05 & 2.23 \\  
                         $n=500,k=5$
                         & PG-BB & 5.4557047e+00 & 9.13e-06 & 8.78e-05 & 9.95 \\  
                         $n=500,k=10$
                         & PG-BB & 6.4795547e+00 & 2.83e-05 & 7.98e-05 & 20.18 \\ 
                         \hline
                        $n=1000,k=2$ 
                         & PG-BB & 3.1810773e+00 & 1.39e-06 & 8.06e-05 & 6.43 \\  
                         $n=1000,k=5$ 
                        & PG-BB & 5.5793546e+00 & 1.62e-05 & 4.76e-05 & 4.34 \\ 
                        $n=1000,k=10$ 
                         & PG-BB & 6.6260901e+00 & 6.64e-05 & 2.62e-05 & 25.73 \\
                        \hline
                        $n=5000,k=2$
                         & PG-BB & 3.2227600e+00 & 4.88e-05 & 8.56e-05 & 66.61 \\ 
                        $n=5000,k=5$
                        & PG-BB & 5.6587667e+00 & 2.13e-05 & 8.60e-05 & 128.42 \\ 
                        $n=5000,k=10$
                        & PG-BB & 6.7409254e+00 & 1.52e-05 & 9.62e-05 & 292.78 \\ 
                        \hline
					\bottomrule
				\end{tabular*}
			\end{minipage}
		\end{center}
	\end{table}

The numerical results are presented in Table~\ref{Table_FPCA}. As shown in the table, PG-BB successfully solves all the problems to the required accuracy, demonstrating the flexibility of our constraint dissolving approach. In particular, the constraint dissolving approach is capable of  effectively handling more complicated constraints, such as the spectral constraint $\norm{P}_2 \leq 1$ in \eqref{Example_FPCA1}.

    \section{Conclusion}
    Constrained optimization problems of the form \eqref{Prob_Ori} have been thought to be distinct from minimization problems over $\X$. Therefore, it is challenging to apply advanced optimization approaches for optimization over $\X$ to solve constrained optimization problems \eqref{Prob_Ori}. 

    The main contribution of this paper is to develop an appealing approach to solve \eqref{Prob_Ori} via the constraint dissolving problem \eqref{Prob_Pen}, which minimizes the constraint dissolving function $h$ over $\X$. With the extension of rCRCQ on \eqref{Prob_Ori}, we prove the equivalence between \eqref{Prob_Ori} and \eqref{Prob_Pen}. Specifically, we prove that \eqref{Prob_Ori} and \eqref{Prob_Pen} have the same first-order stationary points, second-order stationary points, SOSC points, and strong SOSC points, in a neighborhood of the feasible region $\K$. Additionally, any point $x \in \K$ satisfying strict complementarity for \eqref{Prob_Ori} also satisfies strict complementarity for \eqref{Prob_Pen}. 

    Moreover, we discuss how to construct the constraint dissolving mapping $\A$ for \eqref{Prob_Pen} and provide explicit formulations of $\A$ for a 
    variety of commonly encountered convex constraint sets $\X$. As \eqref{Prob_Pen} inherits various desirable theoretical properties from \eqref{Prob_Ori}, various approaches for optimization over $\X$ can be directly applied to solve \eqref{Prob_Ori} through \eqref{Prob_Pen} with theoretical guarantees directly inherited from existing results. For example, the projected gradient method can be directly implemented to solve \eqref{Prob_Ori} through \eqref{Prob_Pen}, for which some recent results \cite{lan2024projected} on its convergence and worst-case complexity can be directly applied. Furthermore, we perform preliminary numerical experiments on nonnegative sparse PCA, nonconvex QP with ball constraints, and fair PCA. The numerical results demonstrate that directly applying the solvers for optimization over $\X$ through \eqref{Prob_Pen} can achieve better efficiency than directly applying state-of-the-art constrained optimization solvers to \eqref{Prob_Ori}. The results thus demonstrate the promising potential of the approach of solving \eqref{Prob_Ori} via the proposed constraint dissolving problem \eqref{Prob_Pen}.

    \section*{Acknowledgments}
    The authors express their gratitude to Professor Xin Liu and Dr. Kuangyu Ding for their valuable comments on constraint dissolving approaches.

	\bibliographystyle{plain}
	\bibliography{ref}

\end{document}